\documentclass[11pt,reqno]{amsart}
\usepackage{fullpage,a4wide}

\usepackage{graphicx}
\usepackage[draft]{hyperref}
\usepackage{amsmath,amsopn,amssymb,amsfonts,stmaryrd}
\usepackage{verbatim}
\usepackage{amsthm}
\usepackage{mathtools}
\usepackage{color}
\usepackage{enumitem}
\usepackage[framemethod=TikZ]{mdframed}
\usepackage{bbm}
\usepackage{longtable}
\usepackage{mathrsfs}
\usepackage{booktabs}
\usepackage{caption}
\usepackage{bm}
\usepackage{tensor}
\usepackage{cleveref}
\usepackage{multirow}
\usepackage{wasysym}

\renewcommand{\H}{\mathbb{H}}

\newcommand{\sgn}{\mbox{sgn}}

\newcommand{\SL}{\mathrm{SL}}

\newcommand{\N}{\mathbb N}
\newcommand{\C}{\mathbb C}
\newcommand{\Q}{\mathbb Q}

\theoremstyle{plain}
\newtheorem{thm}{Theorem}[section]
\newtheorem{cor}[thm]{Corollary}
\newtheorem{lem}[thm]{Lemma}

\newtheorem*{rem}{Remark}

\theoremstyle{definition}

\numberwithin{equation}{section}

\renewcommand{\sgn}{\textnormal{sgn}}

\def\z{\zeta}

\def\vth{\vartheta}

\def\t{\tau}

\def\z{\zeta}

\def\vth{\vartheta}

\def\t{\tau}

\newcommand{\re}{{\rm Re}}
\newcommand{\im}{{\rm Im}}

\renewcommand{\sgn}{{\rm sgn}}
\newcommand{\R}{\mathbb R}
\newcommand{\Z}{\mathbb Z}

\setlist[itemize]{noitemsep, topsep=0pt}

\allowdisplaybreaks

\makeatletter
\newcommand{\vast}{\bBigg@{2}}
\newcommand{\Vast}{\bBigg@{5}}
\makeatother

\renewcommand{\pmod}[1]{\ \left( \mathrm{mod} \, #1 \right)}
\newcommand{\Pmod}[1]{\ ( \mathrm{mod} \, #1 )}

\newcommand{\pent}[1]{P_5\left(#1\right)}

\title{Periodic sign changes for weakly holomorphic $\eta$-quotients}
\author{Kathrin Bringmann}
\address{University of Cologne, Department of Mathematics and Computer Science, Weyertal 86-90, 50931 Cologne, Germany}
\email{kbringma@math.uni-koeln.de}
\author{Guoniu Han}
\address{I.R.M.A., UMR 7501, Universit\'e de Strasbourg et CNRS, 7 rue
Ren\'e Descartes, F-67084 Strasbourg, France}
\email{guoniu.han@unistra.fr}
\author{Bernhard Heim}
\author{Ben Kane}
\address{The University of Hong Kong, Department of Mathematics, Pokfulam, Hong Kong}
\email{bkane@hku.hk}

\makeatletter
\@namedef{subjclassname@2020}{%
	\textup{2020} Mathematics Subject Classification}
\makeatother

	\date{\today}
\begin{document}
	\begin{abstract}
		In this paper, we study sign changes of weakly holomorphic modular forms which are given as $\eta$-quotients. We give representative examples for forms of negative weight, weight zero, and positive weight.
	\end{abstract}
	\maketitle
	\section{Introduction and statement of results}

Let $m\in\N$ and $\delta_{\ell}\in \Z$ for $1\leq \ell\leq m$.
 We define
	\[
		\prod_{\ell=1}^{m} \left(q^\ell;q^\ell\right)_{\infty}^{\delta_\ell}=:\sum_{n\ge0}C_{1^{\delta_1} 2^{\delta_2}\cdots m^{\delta_m}}(n)q^n,
	\]
	where $(a;q)_{n}:=\prod_{j=0}^{n-1}(1-aq^{j})$ for $n\in\N_0\cup\{\infty\}$ is the {\it$q$-Pochhammer symbol}. The signs
	\[
		s_{1^{\delta_1} 2^{\delta_2}\cdots m^{\delta_m}}(n):=\sgn\left(C_{1^{\delta_1} 2^{\delta_2}\cdots m^{\delta_m}}(n)\right)
	\]
were investigated by a number of authors. For example, letting $M(a,c;n)$ denote the number of partitions of $n$ with crank $\equiv a\pmod{c}$, Andrews--Lewis \cite[Conjecture 2]{AndrewsLewis} conjectured that\footnote{We have $M(0,3;3n+2)=M(1,3;3n+2)$ for $n\in\{4,5\}$ and
	$M(0,3;5)>M(1,3;5)$.}
\begin{align*}
M(0,3;3n)&>M(1,3;3n)\text{ for $n\in\N$},\\
M(0,3;3n+1)&<M(1,3;3n+1)\text{ for $n\in\N$},\\
M(0,3;3n+2)&<M(1,3;3n+2)\text{ for }n\in\N\setminus\{1,4,5\}.%
\end{align*}
Since $M(0,3;n)-M(1,3;n)=C_{1^23^{-1}}(n)$, this conjecture can be repackaged as
\[
s_{1^23^{-1}}(n)=\begin{cases} 1&\text{if }3\mid n\text{ or }n=5,\\ 0&\text{if }n\in\{14,17\},\\ -1&\text{otherwise}.\end{cases}
\]
This conjecture was proven by Kane \cite[Corollary 2]{KaneResolution}. To give another example, Andrews \cite[Theorem 2.1]{AndrewsBorwein} proved that for a prime $p$ the signs $s_{1^1p^{-1}}(n)$ of the coefficients of the infinite Borwein products $\frac{(q;q)_{\infty}}{(q^p;q^p)_{\infty}}$ are periodic in $n$ with period $p$. In this paper, we investigate other cases where $s_{1^{\delta_1} 2^{\delta_2}\cdots m^{\delta_m}}(n)$ is periodic with some period $M\in\N$ (i.e., $s_{1^{\delta_1} 2^{\delta_2}\cdots m^{\delta_m}}(n)$ only depends on $n$ (mod $M$)). Techniques used to show these results vary depending on the sign of $\sum_{\ell=1}^{m}\delta_\ell$, so we give representative cases for each of the possibilities for the sign.  We first consider a case with $\sum_{\ell=1}^{m}\delta_\ell<0$ and $M=5$.
	\begin{thm}\label{T:Nw}
		We have
		\[
		s_{1^15^{-2}}(n)=
		\begin{cases}
			1&\text{if }n\equiv 0\pmod{5},\\
			-1&\text{if }n\equiv 1,2\pmod{5},\\
			0&\text{if }n\equiv 3,4\pmod{5}.
		\end{cases}
		\]
	\end{thm}
\begin{rem}
As noted by Wang \cite[(1.4)]{Wang}, Andrews's proof of \cite[Theorem 2.1]{AndrewsBorwein} implies that the signs $s_{1^15^{-1}}(n)$ of the Fourier coefficients of $\frac{(q;q)_{\infty}}{(q^5;q^5)_{\infty}}$ are periodic with period $5$, up to some exceptional $n$ satisfying $s_{1^15^{-1}}(n)=0$. Using this fact, a straightforward argument shows that the signs of $s_{1^15^{-2}}(n)$ of the Fourier coefficients of $\frac{(q;q)_{\infty}}{(q^5;q^5)_{\infty}^2}=\frac{(q;q)_{\infty}}{(q^5;q^5)_{\infty}}\frac{1}{(q^5;q^5)_{\infty}}$ satisfy the same property. Proving that none of the coefficients in the congruence classes $n\equiv 0,1,2\pmod{5}$ vanish requires a slightly more delicate analysis, however. Since we are only interested in purely periodic signs in this paper, we still investigate this case to demonstrate how to use the methods in this paper.
\end{rem}
\Cref{T:Nw} also has a combinatorial interpretation. To state it, for $n\in\N$, let $p_2(n)$ be the number of $2$-colored partitions of $n$, setting $p_2(x):=0$ for $x\notin\N_0$, and for $m\in\Z$ let $\pent{m}:=\frac{3m^2-m}{2}$ be the $m$-th generalized pentagonal number.
\begin{cor}\label{cor:2-color}
For $n\in\N$, we have
\[
\sum_{m\in\Z} (-1)^m p_2\left(\frac{n-\pent{m}}{5}\right)\begin{cases} >0&\text{if }n\equiv 0\pmod{5},\\
<0&\text{if }n\equiv 1,2\pmod{5},\\
=0&\text{if }n\equiv 3,4\pmod{5}.
\end{cases}
\]
\end{cor}
\rm

	We next treat a case with $\sum_{\ell=1}^m\delta_\ell=0$ and $M=4$.
	\begin{thm}\label{T:0w}
		We have
		\[
		s_{1^12^24^{-3}}(n)=\begin{cases}
			1&\text{if }n\equiv 0,3\pmod{4},\\
			-1&\text{if }n\equiv 1,2\pmod{4}.
		\end{cases}
		\]
	\end{thm}
\Cref{T:0w} has an interesting interpretation
in terms of ways to write $n$ as a sum of squares, triangular numbers $T_n:=\frac{n(n+1)}{2}$, and generalized pentagonal numbers. Namely, $C_{1^12^24^{-3}}(n)$ counts certain weighted solutions to the equation
\begin{equation}\label{eqn:formulasquaretrianglepentagonal}
n=2\sum_{j=1}^3 n_j^2 + \sum_{j=1}^{m_1} T_{a_j} + \sum_{j=1}^{m_2} T_{b_j} + \sum_{j=1}^{m_3} \pent{c_j}
\end{equation}
with $\bm{n}\in\Z^3$, $\bm{a}\in\N^{m_1}$, $\bm{b}\in\N^{m_2}$, and $\bm{c}\in(\Z\setminus\{0\})^{m_3}$ (with $m_j\in\N_0$ arbitrary). For $n\in\N$, let $\alpha(n)$ be the number of solutions $(\bm{n},\bm{a},\bm{b},\bm{c})\in \Z^3\times \N^{m_1}\times \N^{m_2}\times(\Z\setminus\{0\})^{m_3}$ to the equation \eqref{eqn:formulasquaretrianglepentagonal} weighted by $(-1)^{\sum_{j=1}^3n_j+m_1+m_2+m_3+\sum_{j=1}^{m_3}c_j}$.
\begin{cor}\label{cor:SquaresTriangularPentagonal}
We have
		\[
		\sgn\left(\alpha(n)\right)=\begin{cases}
			1&\text{if }n\equiv 0,3\pmod{4},\\
			-1&\text{if }n\equiv 1,2\pmod{4}.
		\end{cases}
		\]
\end{cor}
	We finally consider a case with $\sum_{\ell=1}^m\delta_\ell>0$ and $M=9$.
	\begin{thm}\label{T:Pw}
		We have
		\[
		s_{1^93^{-5}}(n)=\begin{cases}
			1&\text{if }n\equiv 0,2,5,6,8\pmod{9},\\
			-1&\text{if }n\equiv 1,3,4,7\pmod{9}.
		\end{cases}
		\]
	\end{thm}
	The paper is organized as follows. In Section \ref{sec:prelim}, we give preliminary facts about modular forms and their Fourier coefficients, Kloosterman sums, and Bessel functions. In Section \ref{sec:T:Nw}, we prove \Cref{T:Nw}. Section \ref{sec:T:0w} is devoted to the proof of \Cref{T:0w}. In Section \ref{sec:T:Pw} we show \Cref{T:Pw}. Finally, further related conjectures are given in Appendix \ref{sec:Conjectures}.
	\section*{Acknowledgements}

	The first author has received funding from the European Research Council (ERC) under the European Union's Horizon 2020 research and innovation programme (grant agreement No. 101001179).
	The fourth author was supported by grants from the Research Grants Council of the Hong Kong SAR, China (project numbers HKU 17314122, HKU 17305923).
	\section{Preliminaries}\label{sec:prelim}
\subsection{Modular forms}
We set
\[
\varepsilon_{d}:=\begin{cases} 1 &\text{if }d\equiv 1\pmod{4}\hspace{-1pt},\\ i&\text{if }d\equiv 3\pmod{4}.\hspace{-1pt}\end{cases}
\]
Suppose that  $\kappa\in\frac{1}{2}\Z$ and $\Gamma$ is a congruence subgroup of $\SL_2(\Z)$, with $\Gamma\subseteq \Gamma_0(4)$ if $\kappa\in \Z+\frac12$, containing $T:=\left(\begin{smallmatrix}1&1\\ 0 &1\end{smallmatrix}\right)$. For $\gamma=\left(\begin{smallmatrix}a&b\\c&d\end{smallmatrix}\right)\in \Gamma$, the weight $\kappa$ \begin{it}slash operator\end{it} is defined by
\[
F\big|_{\kappa}\gamma(\tau):=
\begin{cases}
\left( \frac cd \right)^{2k} \varepsilon_d^{2\kappa}(c\tau+d)^{-\kappa} F(\gamma\tau)&\text{if }\kappa\in\Z+\frac12,\\
(c\tau+d)^{-\kappa} F(\gamma\tau)&\text{if }\kappa\in\Z.
\end{cases}
\]
Here $(\frac\cdot\cdot)$ is the extended Legendre symbol. A holomorphic function $F:\H\to\C$ is called a \begin{it}weight $\kappa$ weakly holomorphic modular form on $\Gamma$ with character $\chi$\end{it} if for every $\gamma\in\Gamma$ we have
\[
F|_{\kappa}\gamma  = \chi(d) F
\]
and for every $\gamma\in\SL_2(\Z)$ there exists $n_0\in\Q$ such that
\begin{equation}\label{eqn:weakgrowth}
(c\tau+d)^{-\kappa} F(\gamma \tau) e^{2\pi i n_0\tau}
\end{equation}
is bounded as $\tau\to i\infty$. We call the equivalence classes of $\Gamma\backslash (\Q\cup\{i\infty\})$ the \begin{it}cusps of $\Gamma$\end{it}. We abuse notation and also call representatives $\varrho\in\Q$ of elements of $\Gamma\backslash (\Q\cup\{i\infty\})$ cusps. For a cusp $\varrho$ we choose $\gamma_{\varrho}\in\SL_2(\Z)$ such that $\gamma_{\varrho}(i\infty)=\varrho$. If $F$ is a weakly holomorphic modular form of weight $\kappa$ on $\Gamma$ with some character $\chi$, then $F_{\varrho}(\tau):=(c\tau+d)^{-\kappa}F(\gamma_{\varrho}\tau)$ is invariant under $T^{\sigma_{\varrho}}$ for some $\sigma_{\varrho}\in\N$. Hence $F$ has a Fourier expansion (with $q:=e^{2\pi i\tau}$)
\[
F_{\varrho}(\tau)= \sum_{n\gg-\infty} c_{F,\varrho}(n) q^{\frac{n}{\sigma_{\varrho}}}.
\]
Note that there are only finitely many negative $n$ because of \eqref{eqn:weakgrowth}. We drop $\varrho$ from the notation if $\varrho=i\infty$. We call the terms in the Fourier expansion with $n<0$ the \begin{it}principal part\end{it} of $F$ at the cusp $\varrho$.

	\subsection{Special modular forms}\label{sec:special}
	Recall the transformation law of the partition generating function  $P(q):=\sum_{n\ge0} p(n)q^n =\frac1{(q;q)_\infty}$, where $p(n)$ denotes the number of partitions of $n$, and the {\it Dedekind $\eta$-function} $\eta(\t):=q^{\frac{1}{24}}(q;q)_{\infty}$. We then take $\tau=\frac1k(h+iz)$ with $z\in\C$ with $\re(z)>0$, so that $q=e^{\frac{2\pi i}{k}(h+iz)}$. Here $h,k\in\N$ satisfy $0\le h<k$ and $\gcd(h,k)=1$, and for $hh'\equiv-1\Pmod k$ we set $q_1:=e^{\frac{2\pi i}{k}(h'+\frac iz)}$. Let $\omega_{h,k}$ be defined through
	\begin{equation}\label{E:P}
		P(q) = \omega_{h,k}\sqrt ze^{\frac{\pi}{12k}\left(\frac1z-z\right)}P(q_1).
	\end{equation}
	Then we have (see \cite[equation (5.2.4)]{A1})
	\begin{equation}\label{E:MultP}
		\omega_{h,k} =
		\begin{cases}
			\left(\frac{-k}{h}\right)e^{-\pi i\left(\frac14(2-hk-h)+\frac{1}{12}\left(k-\frac1k\right)\left(2h-h'+h^2h'\right)\right)} & \text{if }h\text{ is odd},\\
			\left(\frac{-h}{k}\right)e^{-\pi i\left(\frac14(k-1)+\frac{1}{12}\left(k-\frac1k\right)\left(2h-h'+h^2h'\right)\right)} & \text{if }k\text{ is odd}.
		\end{cases}
	\end{equation}
With $[a]_b$ the inverse of $a\Pmod b$ for $\gcd(a,b)=1$, $P$ has the following transformation. %
\begin{lem}\label{lem:P|Vd}
If $\gcd(d,k)=g$, then
\[
P\left(q^d\right)= \omega_{\frac{d}{g}h,\frac{k}{g}}\sqrt{\frac{dz}{g}} e^{\frac{\pi g}{12k}\left(\frac{g}{dz}-\frac{dz}{g}\right)} P\left(e^{\frac{2\pi i g}{k}\left(\left[\frac{d}{g}\right]_{\frac{k}{g}} h'+\frac{i g}{dz}\right)}\right).
\]
\end{lem}

Moreover, since $\eta(\t)=\frac{q^{\frac{1}{24}}}{P(q)}$, \eqref{E:P} implies that, for any $h,h' \in\Z$ with $hh'\equiv -1\pmod{k}$,
\begin{equation*}
	\eta\left(\frac1k(h+iz)\right)  = \omega_{h,k}^{-1}e^{\frac{\pi i\left(h-h'\right)}{12k}}e^{-\frac{\pi i}{4}} (-iz)^{-\frac12}   \eta\left(\frac1k\left(h'+\frac iz\right)\right).
\end{equation*}
In particular, $\eta(24\t)$ is a weight $\frac{1}{2}$ modular form on $\Gamma_0(576)$ with character $\chi_{12}$, where $\chi_D(n):=(\frac{D}{n})$ (see \cite[Corollary 1.62]{O}).

	\subsection{Zuckerman and exact formulas}
	Let $F$ be a weakly holomorphic modular form of weight $\kappa\in -\frac{1}{2}\N_0$ for some congruence subgroup $\Gamma\subseteq\SL_2(\Z)$.  Suppose that for $\gamma = \begin{psmallmatrix}a&b\\c&d\end{psmallmatrix}\in\SL_2(\Z)$ with $c\neq 0$ and $\varrho=\frac{a}{c}\in\Q$ we have the transformation law
	$$
		F(\gamma \tau) = \chi(\gamma) (c \tau + d)^{\kappa} F_{\varrho}(\tau).
	$$
Now let $\gamma=\gamma_{h,k}\in\SL_2(\Z)$ with $a=h$ and $c=k$. Note that $h'=-d$ satisfies the congruence $hh' \equiv -1 \Pmod k$. Taking $\tau = \frac1k(h'+\frac iz)$, we obtain the transformation
	$$
		F\left(\frac1k(h+iz)\right) = \chi(\gamma_{h,k})(-iz)^{-\kappa}   F_{\varrho}\left(\frac1k\left(h'+\frac iz\right)\right).
	$$
	Let $F$ have the Fourier expansion at $i\infty$
	\[
		F(\tau) = \sum_{n\gg-\infty} a(n)q^{n+\alpha}
	\]
	and Fourier expansions at each rational number $0 \leq \frac{h}{k} < 1$ (with $\gcd(h,k)=1$)
	\begin{equation*}\label{2.5}
		(k\tau-h')^{-\kappa}F\left(\gamma_{h,k}\tau\right) = \sum_{n \gg -\infty} a_{h,k}(n) q^{\frac{n + \alpha_{h,k}}{c_{k}}}.
	\end{equation*}
	Furthermore, let $I_\alpha$ denote the usual $I$-Bessel function.
	In this framework, the relevant theorem of Zuckerman \cite[Theorem 1]{Z}, which was extended to a larger class of functions which include weight zero weakly holomorphic modular forms by Ono and the first author \cite[Theorem 1.1]{BringmannOno}, may be stated as follows.

	\begin{thm}\label{Thm: Zuckerman}
		Assume the notation and hypotheses above. If $n + \alpha > 0,$ then we have
		\begin{multline*}
			a(n) =  2\pi (n+\alpha)^{\frac{\kappa-1}{2}} \sum_{k\ge1} \dfrac{1}{k} \sum_{\substack{0 \leq h < k \\ \gcd(h,k) = 1}}\chi(\gamma_{h,k}) e^{- \frac{2\pi i (n+\alpha) h}{k}}
			\\
			\times \sum_{m+\alpha_{h,k} < 0} a_{h,k}(m) e^{ \frac{2\pi i}{k c_{k}} (m + \alpha_{h,k}) h' } \left( \dfrac{\lvert m +\alpha_{h,k} \rvert }{c_{k}} \right)^{ \frac{1 - \kappa}{2}} I_{-\kappa+1}\left( \dfrac{4\pi}{k} \sqrt{\dfrac{(n + \alpha)\lvert m +\alpha_{h,k} \rvert}{c_{k}}} \right).
		\end{multline*}
	\end{thm}
	\subsection{Kloosterman sums}\hspace{0cm}
	We define the Kloosterman sums
	\begin{equation*}
		K_k(n,m) := \sum_{h\Pmod k^*} e^{\frac{2\pi i}{k}\left(nh+mh'\right)}.
	\end{equation*}
We require  Weil's \cite{Weil} bound
  (see \cite[(11.16)]{IwaniecKowalski} for the statement in this form).
	\begin{lem}\label{L:Klooster}
		We have, for $k\in\N$ and $n,m\in\Z$,
		\[
			K_k(n,m) \le \sqrt{\gcd(n,m,k)} d(k)\sqrt{k},
		\]
	where $d(k)$ denotes the number of divisors of $k$.
	\end{lem}
	\subsection{Bessel functions}\hspace{0cm}
	We require certain bounds for the $I$-Bessel functions. Upper bounds for $I_{\kappa}(x)$ are well-known and may be found, for example, in \cite[Lemma 2.2 (1) and (3)]{BKRT}. A lower bound for $I_{1}(x)$ for $x$ sufficiently large was given in \cite[Lemma 2.4]{BHK}, and a lower bound for $I_{\frac{3}{2}}(x)$  follows by a similar argument using \cite[(10.47.7) and (10.49.9)]{NIST}.
	\begin{lem}\label{L:Bessel}\hspace{0cm}
		\begin{enumerate}[wide,labelindent=0pt,label=\rm(\arabic*)]
			\item For $0\le x<1$, $\kappa\in\R$ with $\kappa>-\frac12$, we have
			\begin{equation*}
				I_{\kappa}(x) \le \frac{2^{1-\kappa}x^{\kappa}}{\Gamma(\kappa+1)}.
			\end{equation*}
			\item For $x\geq 1$ and $\kappa\in\R$ with $\kappa>-\frac{1}{2}$, we have
			\begin{equation*}
				I_{\kappa}(x) \le \sqrt{\frac2{\pi x}} e^x.
			\end{equation*}
			\item For $\kappa\in \{1,\frac{3}{2}\}$ and $x\ge 3$, we have
			\begin{equation*}
				I_{\kappa}(x) \ge \frac{e^x}{4\sqrt x}.
			\end{equation*}
		\end{enumerate}
	\end{lem}

 	\section{Proof of \Cref{T:Nw} and \Cref{cor:2-color}}\label{sec:T:Nw}
	For ease of notation, we abbreviate $c_1(n):=C_{1^15^{-2}}(n)$, $s_1(n):=\sgn(c_1(n))$, and set
	\begin{equation}\label{eqn:f1^15^{-2}}
		f_1(q) := \frac{(q;q)_\infty}{\left(q^5;q^5\right)^2} = \frac{P\left(q^5\right)^2}{P(q)} = \sum_{n\ge0} c_1(n) q^n.
	\end{equation}
	\subsection{The case $n\equiv3,4\Pmod5$}
	 Using (see Theorem 1.60 of \cite{O})
	 \begin{equation*}
	 	(q;q)_\infty = \sum_{n\in\Z} (-1)^n q^\frac{n(3n+1)}2,
	 \end{equation*}
	 we see that the Fourier coefficients of $f_1$ are not supported on exponents that are congruent to $3,4\Pmod5$.	This gives \Cref{T:Nw} for $n\equiv3,4\Pmod5$.
	\subsection{The exact formula}
We next use \Cref{Thm: Zuckerman} to obtain an exact formula for $c_1(n)$. To state the result, set
\[
\chi_1(h,k):=\frac{\omega_{h,\frac k5}^2}{\omega_{h,k}}.
\]
Note that $\chi_1(h+k,k)=\chi_1(h,k)$, so $\chi_1$ only depends on $h$ (mod $k$).
\begin{lem}\label{lem:1^15^{-2}exact}
For $n\in\N$, we have
	\begin{equation*}%
		c_1(n) = \frac{2\cdot3^{\frac{3}{4}}\pi}{\left(8n-3\right)^{\frac{3}{4}}} \sum_{\substack{k\ge1\\5\mid k}} \frac1k \sum_{h\Pmod k^*} \chi_{1}(h,k) e^{-\frac{2\pi i nh}k} I_\frac32\left(\frac{\pi}{2k}\sqrt{3\left(8n-3\right)}\right).
	\end{equation*}
\end{lem}
\begin{proof}
In order to use \Cref{Thm: Zuckerman} for $F_1(\t):=q^{-\frac{3}{8}} f_1(q)$, we need to determine the growth of $F_{1}$ at the cusp $\frac{h}{k}$.
For $5\nmid k$, plugging \eqref{E:P} into \eqref{eqn:f1^15^{-2}} we see that, for $z\to0$,
\[
f_1(q)= 5\frac{\omega_{5h,k}^2}{\omega_{h,k}} \sqrt z e^{-\frac{\pi}{20kz}-\frac{3\pi z}{4k}}\frac{P\left(e^{\frac{2\pi i}{k}\left([5]_k h'+\frac{i}{5z}\right)}\right)^2}{P\left(q_1\right)}\ll \sqrt ze^{-\frac{\pi}{20k}\re\left(\frac{1}{z}\right)}\to 0.
\]
So $F_1$ has no principal part at $\frac{h}{k}$ if $5\nmid k$.

Similarly, for $5\mid k$ we obtain
	\begin{equation*}
		f_1(q) = \chi_1(h,k) \sqrt z e^{\frac{3\pi}{4k}\left(\frac1z-z\right)} (1+O(q_1)).
	\end{equation*}
The claim then follows by \Cref{Thm: Zuckerman}.
\end{proof}

We split the formula in \Cref{lem:1^15^{-2}exact} into a main term
\[
M_1(n):=\frac{4\cdot3^\frac34\pi }{5\left(8n-3\right)^\frac34} I_\frac32\left(\frac{\pi}{10}\sqrt{3(8n-3)}\right)\left( \cos\left(\frac{4\pi n}5\right)-\cos\left(\frac{2\pi}5(n-2)\right)\right)
\]
and an error term
\[
	E_1(n):=\frac{2\cdot3^\frac34\pi}{(8n-3)^{\frac{3}{4}}} \sum_{\substack{k>5 \\5\mid k}} \frac1k \sum_{h\Pmod k^*} \chi_1(h,k) e^{-\frac{2\pi i nh}k} I_\frac32\left(\frac\pi {2k}\sqrt{3(8n-3)}\right).
\]
\begin{lem}\label{lem:main+error:1^15^{-2}}
We have
\[
c_1(n) =M_1(n)+E_1(n).
\]
Moreover, if $|M_1(n)|>|E_1(n)|$, then $s_1(n)$ agrees with the value claimed in \Cref{T:Nw}.
\end{lem}
\begin{proof}
The first identity follows from a straightford calculation showing that $M_1(n)$ is the term with $k=5$ of \Cref{lem:1^15^{-2}exact}.

 Now suppose that $|M_1(n)|>|E_1(n)|$. Then the first identity implies that
\begin{equation*}%
s_1(n)=\sgn\left(M_1(n)\right) = \begin{cases}
			1&\text{if }n\equiv 0\pmod{5},\\
			-1&\text{if }n\equiv 1,2\pmod{5}.
		\end{cases}
\end{equation*}
	This matches what is claimed in \Cref{T:Nw}.
\end{proof}

	\subsection{Bounding the main and error terms}

We next bound the main term and error term from Lemma \ref{lem:main+error:1^15^{-2}}.
\begin{lem}\label{lem:1^15^{-2}ErrorBound}\hspace{0cm}
\begin{enumerate}[wide,labelindent=0pt,label=\rm(\arabic*)]
\item
For $n\equiv 0,1,2\pmod{5}$, we have
\[
		\left|M_1(n)\right|\geq \frac{4\cdot3^\frac34\pi}{5(8n-3)^\frac34} I_\frac32\left(\frac\pi{10}\sqrt{3(8n-3)}\right)\left(1-\cos\left(\frac{2\pi}{5}\right)\right).
\]
\item
We have
\[
\left|E_1(n)\right|\leq \frac{4\sqrt6\pi^{\frac32}}{5\left(8n-3\right)^\frac14}+ \frac{3^\frac54\pi^2}{5(8n-3)^\frac14} I_\frac32\left(\frac\pi{20}\sqrt{8n-3}\right).
\]\rm
\end{enumerate}
\end{lem}
\begin{proof}
	\begin{enumerate}[wide,labelindent=0pt]
		\item Directly from the definition, we have
		\[
			\left|M_1(n)\right|\geq \frac{4\cdot3^\frac34\pi}{5(8n-3)^\frac34} I_\frac32\left(\frac\pi{10}\sqrt{3(8n-3)}\right)\min_{n\in\{0,1,2\}} \left|\cos\left(\frac{4\pi n}5\right)-\cos\left(\frac{2\pi}5(n-2)\right)\right|.
		\]
		The minimum occurs for $n=2$, giving the claim.
		\item Making the change of variables $k\mapsto 5k$ and taking the absolute value inside the sum, we conclude that
		\begin{equation}\label{eqn:Ebound1}
			\left|E_1(n)\right|\leq \frac{2\cdot3^\frac34\pi}{\left(8n-3\right)^\frac34} \sum_{k\ge2} I_\frac32\left(\frac\pi{10k}\sqrt{3(8n-3)}\right).
		\end{equation}
		We split the sum over $k$ in \eqref{eqn:Ebound1} into terms with $k$ small and $k$ large.
		Using \Cref{L:Bessel} (1) we conclude that the contribution to \eqref{eqn:Ebound1} from $k>\frac\pi{10}\sqrt{3(8n-3)}$ is bounded by
		\begin{equation}\label{eqn:error2}
			\frac{2\sqrt3\pi^2}{5\sqrt5}  \sum_{k>\frac\pi{10}\sqrt{3(8n-3)}} k^{-\frac32} \leq  \frac{2\sqrt{3} \pi^2}{5\sqrt{5}}\int_{\frac\pi{10}\sqrt{3(8n-3)}}^\infty x^{-\frac32} dx= \frac{4\sqrt6\pi^{\frac32}}{5\left(8n-3\right)^\frac14}.
		\end{equation}

		Estimating the Bessel function against the term with $k=2$, the contribution from $k\le\frac\pi{10}\sqrt{3(8n-3)}$ is bounded by
		\begin{equation}\label{eqn:error1}
			\frac{2\cdot3^\frac34\pi}{\left(8n-3\right)^\frac34} \sum_{2\le k \le\frac\pi{10}\sqrt{3(8n-3)}} I_\frac32\left(\frac\pi{10k}\sqrt{3(8n-3)}\right)\leq \frac{3^\frac54\pi^2}{5(8n-3)^\frac14} I_\frac32\left(\frac\pi{20}\sqrt{8n-3}\right).
		\end{equation}
		Adding \eqref{eqn:error1} to \eqref{eqn:error2} gives the claim.\qedhere
	\end{enumerate}
\end{proof}
	\subsection{Proof of \Cref{T:Nw}}
	We are now ready to prove \Cref{T:Nw}.
	\begin{proof}[Proof of \Cref{T:Nw}]
	Comparing \Cref{lem:1^15^{-2}ErrorBound} (1) with \Cref{lem:1^15^{-2}ErrorBound} (2), for $n\equiv 0,1,2\pmod{5}$ we conclude by Lemma \ref{lem:main+error:1^15^{-2}} that $s_1(n)$ matches the value claimed in \Cref{T:Nw} if
	\begin{multline*}
		\frac{4\cdot3^\frac34\pi}{5(8n-3)^\frac34} I_\frac32\left(\frac\pi{10}\sqrt{3(8n-3)}\right)\left(1-\cos\left(\frac{2\pi}{5}\right)\right)\\
		>\frac{2^{\frac{5}{2}}\sqrt{3}\pi^{\frac32}}{5\left(8n-3\right)^\frac14}+ \frac{3^\frac54\pi^2}{5(8n-3)^\frac14} I_\frac32\left(\frac\pi{20}\sqrt{8n-3}\right).
	\end{multline*}
	Rearranging, this is equivalent to
	\begin{equation*}\label{eqn:ratiotobound}
		\frac{\sqrt{8n-3}}{\left(1-\cos\left(\frac{2\pi}{5}\right)\right) I_\frac32\left(\frac\pi{10}\sqrt{3(8n-3)}\right)}  \left(\frac{\sqrt{2\pi}}{3^{\frac{1}{4}}}+\frac{\sqrt3\pi}{4} I_\frac32\left(\frac\pi{20}\sqrt{8n-3}\right)\right)<1.
	\end{equation*}
	We next use \Cref{L:Bessel} (2) for $n \geq 6$ to bound
		\begin{equation*}
			I_{\frac{3}{2}}\left(\frac{\pi}{20}\sqrt{8n-3}\right)\leq \frac{2\sqrt{10}}{\pi(8n-3)^\frac14} e^{\frac{\pi}{20}\sqrt{8n-3}}.
		\end{equation*}
	Moreover, by \Cref{L:Bessel} (3) for $n\geq 5$, we have
		\begin{equation*}
			I_{\frac{3}{2}}\left(\frac{\pi}{10}\sqrt{3(8n-3)}\right)\geq  \frac{e^{\frac{\pi}{10}\sqrt{3(8n-3)}}}{2\sqrt{2\pi}\, 3^{\frac{1}{4}}(8n-3)^{\frac{1}{4}}} .
		\end{equation*}
	Thus we want
	\begin{align*}
		\frac{(8n-3)^{\frac{3}{4}}}{1-\cos\left(\frac{2\pi}{5}\right)} 2\sqrt{2\pi} 3^{\frac{1}{4}}e^{-\frac{\pi}{10} \sqrt{3(8n-3)}} \left(\frac{\sqrt{2\pi}}{3^{\frac{1}{4}}}+\sqrt{\frac{15}2} (8n-3)^{-\frac{1}{4}}e^{\frac{\pi}{20}\sqrt{8n-3}}\right) < 1.
	\end{align*}
	It can be shown by computer that this holds for $n\ge33$. For $n<33$, we evaluate $s_1(n)$ with a computer to determine that it agrees with the value claimed in \Cref{T:Nw}.
\end{proof}

We next prove \Cref{cor:2-color}.
\begin{proof}[Proof of \Cref{cor:2-color}]
First recall that
\[
\frac{1}{(q^5;q^5)_{\infty}^2}=\sum_{j\geq0} p_2(j)q^{5j}.
\]
Hence \cite[Theorem 1.60]{O} implies that
\[
\frac{(q;q)_{\infty}}{\left(q^5;q^5\right)_{\infty}^2}=\sum_{n\geq0} \sum_{m\in\Z} (-1)^m p_2\left(\frac{n-\pent{m}}{5}\right) q^{n}.
\]
The result follows by \Cref{T:Nw}.\qedhere
\end{proof}
	\section{Proof of \Cref{T:0w} and \Cref{cor:SquaresTriangularPentagonal}}\label{sec:T:0w}
	We abbreviate $c_2(n):=C_{1^12^{2}4^{-3}}(n)$, $s_2(n):=\sgn(c_2(n))$, and let
	\begin{equation*}
		f_2(q) := \frac{(q;q)_\infty\left(q^2;q^2\right)_\infty^2}{\left(q^4;q^4\right)_\infty^3} = \frac{P\left(q^4\right)^3}{P(q)P\left(q^2\right)^2}.
	\end{equation*}
	\subsection{The exact formula}
	We next obtain an exact formula for $c_2(n)$. Let
	\begin{equation*}
		\chi_2(h,k) := \frac{\omega_{h,\frac k4}^3}{\omega_{h,k}\omega_{h,\frac k2}^2}.
	\end{equation*}
	Note that $\chi_2(h+k,k)=\chi_2(h,k)$, so $\chi_2(h,k)$ only depends on $h$ (mod $k$).
	\begin{lem}\label{lem:1^12^24^{-3}exact}
		We have
		\begin{equation*}%
			c_2(n) = \frac{2\sqrt{7}\pi}{\sqrt{24n-7}} \sum_{\substack{k\ge1\\4\mid k}} \frac1k \sum_{h\pmod{k}^*} \chi_2(h,k) e^{-\frac{2\pi i nh}{k}} I_1\left(\frac\pi{6 k}\sqrt{7\left(24n-7\right)}\right).
		\end{equation*}
	\end{lem}
\begin{proof}
	\Cref{lem:P|Vd} and \eqref{E:P} imply that $f_2(q)=O(e^{-\frac{5\pi}{48kz}})$ as $q\to e^\frac{2\pi ih}k$ with $k$ odd. For $2\| k$ \Cref{lem:P|Vd} and \eqref{E:P} again imply that $f_2(q)=O(e^{-\frac\pi{6kz}})$ as $q\to e^\frac{2\pi ih}k$.
	Finally, for $4\mid k$, taking $g=d$ in \Cref{lem:P|Vd} and \eqref{E:P} imply that
	\begin{equation}\label{eqn:fpole}
		f_2(q) = \chi_2(h,k) e^{\frac{7\pi}{12k}\left(\frac1z-z\right)} (1+O(q_1)).
	\end{equation}
	Setting $F_2(\t):=q^{-\frac{7}{24}}f_2(q)$ and plugging \eqref{eqn:fpole} into \Cref{Thm: Zuckerman} gives the claim.\qedhere
\end{proof}

	\subsection{The multiplier system}\label{sec:multiplier}
	We next rewrite the multiplier system $\chi_2$.
	\begin{lem}\label{L:Chi}
		We have
		\begin{equation*}
			\chi_2(h,k) = e^{\frac{2\pi i}{24k}\left(\left(-5\frac {k^{2}}{2}+7\right)h-\left(5\frac {k^{2}}4+7\right)h'\right)},
		\end{equation*}
		where we choose $h'$ to satisfy $hh'\equiv-1\Pmod{8\gcd(k,3)k}$.
	\end{lem}
	\begin{proof}
		 Noting that we may choose $[-h]_{\frac{k}{2}}=[-h]_{\frac{k}{4}}=h'$, we may show that
		\begin{equation*}
			\chi_2(h,k) = e^{2\pi iA},
		\end{equation*}
		where
		\begin{equation*}
			A := \frac1{96k} \left(\left(\left(5k^2+28\right)h^2-5k^2-28\right)h'+\left(-5k^2+56\right)h\right).
		\end{equation*}
		The claim follows by a direct computation distinguishing whether $3\mid k$ or $3\nmid k$.
	\end{proof}
	\subsection{The main term}
Define
\begin{align*}
	M_2(n) &:=\frac{\sqrt7\pi\cos\left(\frac\pi2\left(n+\frac14\right)\right)}{\sqrt{24n-7}} I_1\left(\frac\pi{24}\sqrt{7(24n-7)}\right),\\
	E_2(n) &:=\frac{2\sqrt7\pi}{\sqrt{24n-7}} \sum_{\substack{k\geq 5\\4\mid k}} \frac1k \sum_{h\pmod{k}^*} \chi_2(h,k) e^{-\frac{2\pi i nh}{k}} I_1\left(\frac\pi {6k}\sqrt{7\left(24n-7\right)}\right).
\end{align*}
\begin{lem}\label{lem:main+error:1^12^24^{-3}}
We have
\[
c_2(n) =M_2(n)+E_2(n).
\]
Moreover, if $|M_2(n)|>|E_2(n)|$, then $s_2(n)$ agrees with the claimed value in \Cref{T:0w}.
\end{lem}
\begin{proof}
For the first identity, we need to show that $M_2(n)$ equals the term with $k=4$ of \Cref{lem:1^12^24^{-3}exact}, which is
	\begin{equation*}
		\frac{\sqrt7\pi}{2\sqrt{24n-7}} \sum_{h\in\{1,3\}} \chi_2(h,4) i^{-nh} I_1\left(\frac\pi{24}\sqrt{7(24n-7)}\right).
	\end{equation*}
	Using \Cref{L:Chi}, the sum on $h$ becomes
	\begin{equation*}
		\sum_{h\in\{1,3\}} e^{\frac{2\pi i}{32}(-11h-9h')} e^{-\frac{\pi ihn}2}.
	\end{equation*}
 From this it is not hard to see that the main term is as claimed.

To see the second statement, note that if $|M_2(n)|>|E_2(n)|$, then
	\begin{equation*}
		s_2(n)=\sgn\left(M_2(n)\right)=\sgn\left(\cos\left(\frac\pi2 \left(n+\frac14\right)\right)\right) =
		\begin{cases}
			1 & \text{if $n\equiv0,3\Pmod4$},\\
			-1 & \text{if $n\equiv1,2\Pmod4$}.
		\end{cases}
	\end{equation*}
This matches the claim in \Cref{T:0w}. \qedhere
\end{proof}
	\subsection{Kloosterman sums}
	Let
	\begin{equation*}
		A_k(n) := \sum_{\substack{0\leq h<k\\\gcd(h,k)=1\\hh'\equiv-1\pmod{8\gcd(k,3)k}}} \chi_2(h,k) e^{-\frac{2\pi inh}k}.
	\end{equation*}
	\begin{lem}\label{L:KloostLemma}
		\hspace{0cm}
		\begin{enumerate}[leftmargin=*,label=\rm(\arabic*)]
			\item We have, for $3\mid k$,
			\begin{equation*}
				A_k(n) = \frac1{24} K_{24k} \left(-5\frac{k^2}2 + 7 - 24n, -5\frac{k^2}4 - 7\right).
			\end{equation*}
			\item We have, for $3\nmid k$,
			\begin{equation*}
				A_k(n) = \frac18 K_{8k} \left(\frac13\left(-5\frac{k^2}2 + 7\right) - 8n, \frac13\left(-5\frac{k^2}4 - 7\right)\right).
			\end{equation*}
		\end{enumerate}
	\end{lem}
	\begin{proof}Using \Cref{L:Chi}, we rewrite
		\begin{equation*}
			A_k(n) = \frac1{8\gcd(k,3)} \sum_{h\pmod{8\gcd(k,3)k}^*} e^{\frac{2\pi i}{24k}\left(\left(-5\frac{k^2}2+7-24n\right)h-\left(5\frac{k^2}4+7\right)h'\right)}.
		\end{equation*}
		Distinguishing whether $3\mid k$ or $3\nmid k$ gives the claim.
	\end{proof}

	We can now use \Cref{L:KloostLemma} to explicitly bound $A_k(n)$.
	\begin{lem}\label{lem:Aknbound}
		We have
		\begin{equation*}
			|A_k(n)| \le \frac12\sqrt{\frac{7k}{2}} d(24k).
		\end{equation*}
	\end{lem}
	\subsection{Bounding the error term and finishing the proof of \Cref{T:0w}}
We are now ready to prove \Cref{T:0w}.
\begin{proof}[Proof of \Cref{T:0w}]
	Dividing the error terms $E_2(n)$ by the main term $M_2(n)$, \Cref{lem:main+error:1^12^24^{-3}} implies that \Cref{T:0w} follows if
	\[
		\frac{\left|E_2(n)\right|}{\left|M_2(n)\right|}<1.
	\]
Applying the triangle inequality and then using \Cref{lem:Aknbound}, we need to bound
\begin{align}\nonumber
\frac{\left|E_2(n)\right|}{\left|M_2(n)\right|}&\leq \frac{2}{\left|\cos\left(\frac{\pi}{2}\left(n +\frac14\right)\right)\right| I_1\left(\frac\pi{24}\sqrt{7(24n-7)}\right)}\sum_{\substack{k\ge5\\4\mid k}} \frac{\left|A_k(n)\right|}k I_1\left(\frac\pi{6k}\sqrt{7(24n-7)}\right)\\
&\le \frac{\sqrt7}{2\sqrt{2}\cos\left(\frac{3\pi}{8}\right) I_1\left(\frac\pi{24}\sqrt{7(24n-7)}\right)}  \sum_{k\ge2} \frac{d(4\cdot24k)}{\sqrt k} I_1\left(\frac\pi{24k}\sqrt{7(24n-7)}\right).\label{eqn:E2/M2bound}
\end{align}
	Note that
	\begin{equation*}%
		d(96k) \le d(96) d(k) = 12d(k).
	\end{equation*}
Hence, using \Cref{L:Bessel} (1), the contribution to the right-hand side of \eqref{eqn:E2/M2bound} from $k>\frac\pi{24}\sqrt{7(24n-7)}$ can be bounded by
	\begin{equation*}
		\frac{7\pi\z\left(\frac32\right)^2\sqrt{24n-7}}{4\sqrt{2}\cos\left(\frac{3\pi}{8}\right)I_1\left(\frac\pi{24}\sqrt{7(24n-7)}\right)} .
	\end{equation*}

	We next consider the contribution to the right-hand side of \eqref{eqn:E2/M2bound} from  the terms with $2\leq k\le\frac\pi{24}\sqrt{7(24n-7)}$. Since $k\geq 2$, we may trivially bound
	\[
		I_1\left(\frac\pi{24k}\sqrt{7(24n-7)}\right)\leq I_1\left(\frac\pi{48}\sqrt{7(24n-7)}\right),
	\]
	which we may pull out of the sum on $k$. Since divisors of $k$ may be paired as $(d,\frac{k}{d})$ with $d\leq \frac{k}{d}$, we have at most $\sqrt{k}$ such pairs, and hence we may trivially bound $d(k)\leq 2\sqrt{k}$. Hence the overall contribution to the right-hand side of \eqref{eqn:E2/M2bound} from the terms with $2\leq k\le\frac\pi{24}\sqrt{7(24n-7)}$ may be bounded from above by
	\begin{equation*}
		\frac{7\pi \sqrt{24n-7}}{2\sqrt{2}\cos\left(\frac{3\pi}{8}\right)} \frac{I_1\left(\frac\pi{48}\sqrt{7(24n-7)}\right)}{I_1\left(\frac\pi{24}\sqrt{7(24n-7)}\right)}.
	\end{equation*}
	Combining the two errors, we need
	\begin{equation}\label{eqn:IratioToBound}
		\frac{7\pi\z\left(\frac32\right)^2\sqrt{24n-7}}{4\sqrt{2}\cos\left(\frac{3\pi}{8}\right)I_1\left(\frac\pi{24}\sqrt{7(24n-7)}\right)} + \frac{7\pi  \sqrt{24n-7}}{2\sqrt{2}\cos\left(\frac{3\pi}{8}\right)}  \frac{I_1\left(\frac\pi{48}\sqrt{7(24n-7)}\right)}{I_1\left(\frac\pi{24}\sqrt{7(24n-7)}\right)} < 1.
	\end{equation}

	We next bound \eqref{eqn:IratioToBound} against elementary functions.
	Using \Cref{L:Bessel} (2),(3) for\\ $\frac{\pi}{24}\sqrt{7(24n-7)}\geq 3$, the left-hand side of \eqref{eqn:IratioToBound} may for $n\geq 99$ be estimated against
	\begin{equation*}
		\frac{7^\frac54\pi^\frac32\z\left(\frac32\right)^2(24n-7)^\frac34}{4\sqrt3\cos\left(\frac{3\pi}{8}\right)} e^{-\frac\pi{24}\sqrt{7(24n-7)}} + \frac{14\sqrt{2\pi}\sqrt{24n-7}}{\cos\left(\frac{3\pi}{8}\right)} e^{-\frac\pi{48}\sqrt{7(24n-7)}}<1
	\end{equation*}
	\Cref{T:0w} now follows as it was checked with a computer for $n\leq 98$.
\end{proof}
We conclude the section with a proof of \Cref{cor:SquaresTriangularPentagonal}.
\begin{proof}[Proof of \Cref{cor:SquaresTriangularPentagonal}]
Using \cite[Theorem 1.60]{O}, we can write
\begin{align}
\nonumber \frac{(q;q)_{\infty}(q^2;q^2)_{\infty}^2}{(q^4;q^4)_{\infty}^3} &= \left(\frac{\left(q^2;q^2\right)_\infty^2}{\left(q^4;q^4\right)_\infty}\right)^3 \left(\frac{\left(q^2;q^2\right)_\infty^2}{(q;q)_\infty}\right)^{-2}\frac1{(q;q)_\infty}\nonumber\\
&=\left(\sum_{n\in\Z} (-1)^n q^{2n^2}\right)^3  \frac1{\left(\sum_{n\ge0}q^{\frac{(2n+1)^2-1}{8}}\right)^2} \frac{1}{\sum_{n\in\Z} (-1)^nq^{\frac{(6n+1)^2-1}{24}}}.\label{eqn:squarestrianglespentagonal}
\end{align}
Noting that $\frac{(2n+1)^2-1}{8}=T_n$ and $\frac{(6n+1)^2-1}{24}=\pent{-n}$, the result follows from \Cref{T:0w} after expanding the denominators appearing in \eqref{eqn:squarestrianglespentagonal} as geometric series.
\end{proof}
	\section{Proof of \Cref{T:Pw}}\label{sec:T:Pw}
	\subsection{General transformations}
We abbreviate $c_3(n):=C_{1^93^{-5}}(n)$ and $s_3(n):=\sgn(c_3(n))$. We determine the transformation law of
	\begin{equation*}
		f_3(q) := \frac{(q;q)_\infty^9}{\left(q^3;q^3\right)_\infty^5} = \frac{P\left(q^3\right)^5}{P(q)^9}=\sum_{n\geq 0} c_3(n)q^n.
	\end{equation*}
For $3\mid k$, we set
\begin{equation}\label{eqn:chi3}
\chi_3(h,k):=\frac{\omega_{h,\frac k3}^5}{\omega_{h,k}^9}.
\end{equation}
\Cref{lem:P|Vd}  with $d=g$ and \eqref{E:P} yield
	\begin{equation}\label{eqn:3midk}
		f_3(q) = \chi_3(h,k) z^{-2} e^{\frac\pi{2k}\left(\frac1z-z\right)} f_3(q_1).
	\end{equation}
For $3\nmid k$, we note that we may choose $h'$ with $3\mid h'$, so that
\[
[3]_k h'\equiv \frac{h'}{3}\pmod{k}.
\]
Hence for $3\nmid k$ and $3\mid h'$, \Cref{lem:P|Vd} and \eqref{E:P} yield
	\begin{equation}\label{eqn:3nmidk}
		f_3(q) = \frac{\omega_{3h,k}^5}{\omega_{h,k}^9} 3^\frac52 z^{-2} e^{-\frac{11\pi}{18kz}-\frac{\pi z}{2k}} \frac{P\left(q_1^\frac13\right)^5}{P(q_1)^9}.
	\end{equation}
In particular, this vanishes as  $z\to 0$.
	\subsection{The Circle Method}
	Using the residue theorem, we have
	\begin{equation*}
		c_3(n) = \frac1{2\pi i} \int_{\mathcal C} \frac{f_3(q)}{q^{n+1}} dq =  \sum_{\substack{0\le h<k<N\\\gcd(h,k)=1}} e^{-\frac{2\pi inh}k} \int_{-\vartheta_{h,k}'}^{\vartheta_{h,k}''} f_3\left(e^{\frac{2\pi i}k(h+iz)}\right) e^{\frac{2\pi nz}k} d\Phi = \sum{}_{\substack{\\[6pt]\hspace{-0.05cm}1}} + \sum{}_{\substack{\\[6pt]\hspace{-0.05cm}3}},
	\end{equation*}
	where $z=\frac kn+ik\Phi$, $\vartheta_{h,k}':=\frac1{k(k+k_1)}$, $\vartheta_{h,k}'':=\frac1{k(k+k_2)}$ with $\frac{h_1}{k_1}<\frac hk<\frac{h_2}{k_2}$ are consecutive fractions in the Farey sequence of order $N:=\lfloor\sqrt n\rfloor$ and where $\sum_d$ is the restriction to $k$ with $\gcd(k,3)=d$. Assuming that $d\mid h'$, we may use \eqref{eqn:3midk} and \eqref{eqn:3nmidk} to obtain
	\begin{align*}
		\sum{}_{\substack{\\[6pt]\hspace{-0.05cm}3}} %
		&=\sum_{\substack{0\le h<k<N\\\gcd(h,k)=1\\3\mid k}}
\chi_3(h,k) e^{-\frac{2\pi inh}k} \int_{-\vartheta_{h,k}'}^{\vartheta_{h,k}''} z^{-2} e^{\frac\pi{2k}\left(\frac1z-z\right)}  f_3(q_1)
		e^{\frac{2\pi nz}k} d\Phi,\\[-1.5em]
		\sum{}_{\substack{\\[6pt]\hspace{-0.05cm}1}} %
		&= 3^\frac52 \sum_{\substack{0\le h<k<N\\\gcd(h,k)=1\\3\nmid k}} \frac{\omega_{3h,k}^5}{\omega_{h,k}^9} e^{-\frac{2\pi inh}k} \int_{-\vartheta_{h,k}'}^{\vartheta_{h,k}''} z^{-2} e^{-\frac{11\pi}{18kz}-\frac{\pi z}{2k}}  \frac{P\left(q_1^\frac13\right)^5}{P(q_1)^9} e^{\frac{2\pi nz}k} d\Phi.
	\end{align*}
 We require the well-known bounds
	\begin{equation}\label{eqn:Farey}
	\vartheta_{h,k}',\vartheta_{h,k}''\leq \frac{1}{k(N+1)},\quad \re\left(\frac{1}{z}\right)\geq\frac{k}{2},\quad |z|\geq\frac{k}{n}.%
	\end{equation}

	Next we split-off the principal parts. Note that these only occur if $3\mid k$. We write
	\begin{equation*}
		f_3(q_1) = 1 +  \left(f_3(q_1)-1\right).
	\end{equation*}
	We call the contribution from the constant term $\sum_{3,1}$ and the remaining sum $\sum_{3,2}$. We have
\begin{equation}\label{eqn:CircleMain1}
\sum{}_{\substack{\\[6pt]\hspace{-0.05cm}3,1}} = -i\sum_{\substack{0\le h<k<N\\\gcd(h,k)=1\\3\mid k}} \frac{\chi_3(h,k)}{k} e^{-\frac{2\pi inh}k} \int_{\frac kn-ik\vartheta_{h,k}'}^{\frac kn+ik\vartheta_{h,k}''} z^{-2} e^{\frac\pi{2kz}+\frac{\pi z}{2k}(4n-1)} dz.
\end{equation}
	We now approximate the integral in the following lemma.
	\begin{lem}\label{lem:BesselBoundExplicit}
		We have
		\[
		\left|\int_{\frac kn-ik\vartheta_{h,k}'}^{\frac kn+ik\vartheta_{h,k}''} z^{-2} e^{\frac\pi{2kz}+\frac{\pi z}{2k}(4n-1)} dz - 2\pi i\sqrt{4n-1} I_1\left(\frac\pi k\sqrt{4n-1}\right)\right|\leq  4588024 \frac{n^{\frac{3}{2}}}{k^2}.
		\]
	\end{lem}
	\begin{proof}
		We follow the argument given in \cite[pp. 404--405]{Lehner}. Let $A:=\frac\pi{2k}(4n-1)$ and $B:=\frac\pi{2k}$. From \S 6.22, equation (1) of \cite{W}, we obtain
		\begin{multline*}
			\hspace{-0.2cm}2\pi i\sqrt{4n-1} I_1\left(\frac\pi k\sqrt{4n-1}\right) = \Bigg(\int_{-\infty}^{-\frac kn}+\int_{-\frac kn}^{-\frac kn-ik\vth_{h,k}'}+\int_{-\frac kn-ik\vth_{h,k}'}^{\frac kn-ik\vth_{h,k}'}+ \int_{\frac kn-ik\vth_{h,k}'}^{\frac kn+ik\vth_{h,k}''}+\int_{\frac kn+ik\vth_{h,k}''}^{-\frac kn+ik\vth_{h,k}''}\\
			+ \int_{-\frac kn+ik\vth_{h,k}''}^{-\frac kn}+ \int_{-\frac{k}{n}}^{-\infty} \Bigg) z^{-2}\exp\left(Az+\frac{B}{z}\right) dz =: J_1 + J_2 + \ldots + J_7.
		\end{multline*}
		Note that $J_4$ is the integral appearing on the left-hand side of the lemma, so, subtracting that from the other side, it remains to bound the absolute values of the other integrals.

		For $J_2$, we have $\re(z)=-\frac kn<0$ and $\re(\frac1z)=\frac{\re(z)}{|z|^2}<0$. Thus, we have
		\[
			|J_2| \le \int_{-\frac kn}^{-\frac kn-ik\vth_{h,k}'} |z|^{-2} d|z| = \int_0^{k\vth_{h,k}'} \frac1{\frac{k^2}{n^2}+y^2} dy.
		\]
		We bound this by setting $y=0$ in the integrand and get
		\[
			|J_2|\le \frac{n^2}{k^2}k\vth_{h,k}' \le \frac{n^2}{k^2(N+1)} \leq \frac{n^\frac32}{k^2}.
		\]
		We have the same bound for $|J_6|$.

		For $J_3$, we make the change of variables $z=x-ik\vth_{h,k}'$.  Note that here $-\frac kn\le\mathrm{Re}(z)=x\leq \frac{k}{n}$ and $\re(\frac1z)\le 4k$. Thus, on $J_3$ we have
		\[
			\left|\exp\left(Az+\frac Bz\right)\right| \le \exp(4\pi),
		\]
		so
		\begin{equation*}
			|J_3| \le \exp(4\pi) \int_{-\frac kn}^\frac kn \frac{1}{|x-ik\vth_{h,k}'|^{2}} dx \le \exp(4\pi) \int_{-\frac kn}^\frac kn \frac{1}{x^2+(k\vth_{h,k}')^2} dx.
		\end{equation*}
		We bound the integral by setting $x=0$ in the integrand and get $|J_3|\le8\exp(4\pi) k$. We bound $J_5$ in the same way.

		Finally, we combine
\[
J_1+J_7=0.
\]
		Using that $1\le k\le\sqrt n$, we easily obtain the claim.
	\end{proof}
	Plugging \Cref{lem:BesselBoundExplicit} into \eqref{eqn:CircleMain1}, we obtain
	\begin{equation}\label{eqn:CircleMain}
		\sum{}_{\substack{\\[6pt]\hspace{-0.05cm}3,1}} = 2\pi\sqrt{4n-1}\sum_{\substack{0\le h<k<N\\\gcd(h,k)=1\\3\mid k}} \frac{\chi_3(h,k)}{k} e^{-\frac{2\pi inh}k}  I_{1}\left(\frac{\pi}{k}\sqrt{4n-1}\right) + E(n),
	\end{equation}
	where $E(n)$ may be bounded against
	\begin{equation}\label{eqn:Enbound}
		|E(n)| \le  4588024n^\frac32 \sum_{\substack{1\le k<N\\3\mid k}} \frac1{k^2}\le \frac{4588024}{9} \zeta(2) n^\frac32 = \frac{2294012\pi^2}{27} n^\frac32.
	\end{equation}

We define
\begin{align*}
	M_3(n):&=\frac{2\pi}{3\gcd(n,3)} \alpha_n \sqrt{4n-1} I_{1}\left(\frac{\pi\sqrt{4n-1}}{3\gcd(n,3)}\right),\\
	E_3(n):&= 2\pi\sqrt{4n-1}\sum_{\substack{0\le h<k<N\\\gcd(h,k)=1\\3\mid k\\ k\geq 3\ell_n}} \frac{\chi_3(h,k)}{k} e^{-\frac{2\pi inh}k}  I_{1}\left(\frac{\pi}{k}\sqrt{4n-1}\right),
\end{align*}
where
\begin{align*}
	\alpha_n:&=\begin{cases}
		-2\sin\left(\frac{2\pi n}{3}\right)&\text{if }3\nmid n,\\
		2\left(\sin\left(\frac{2\pi}9(4-2n)\right)-\sin\left(\frac{2\pi}9(5-n)\right)-\sin\left(\frac{2\pi}9(5-4n)\right)\right)&\text{if }3\mid n,
	\end{cases}\\
	\ell_n:&=\begin{cases}2&\text{if }3\nmid n,\\ 4&\text{if }3\mid n.\end{cases}
\end{align*}
We obtain a formula for $c_3(n)$.
\begin{lem}\label{lem:1^93^{-5}almostexact}
We have
\[
c_3(n)=M_3(n)+E_3(n) +E(n)+\sum{}_{\substack{\\[6pt]\hspace{-0.05cm}3,2}}+\sum{}_{\substack{\\[6pt]\hspace{-0.05cm}1}}.
\]
Moreover, if $|M_3(n)|>|E_3(n)| +|E(n)|+|\sum_{3,2}| + |\sum_1|$, then $s_3(n)=\sgn(\alpha_n)$, which agrees with the claimed value of $s_3(n)$ in \Cref{T:Pw}.
\end{lem}
\begin{proof}
Recall \eqref{eqn:CircleMain}. For $3\nmid n$ the first identity is equivalent to showing that $M_3(n)$ is the term with $k=3$ of the sum in \eqref{eqn:CircleMain}, while for $3\mid n$ it is equivalent to showing that the terms with $k=3$ and $k=6$ vanish and $M_3(n)$ equals the term with $k=9$ in \eqref{eqn:CircleMain}.

We first evaluate the term coming from $k=3$. Using \eqref{E:MultP} and \eqref{eqn:chi3} of $\chi_3(h,k)$, the sum on $h$ for this term equals
	\begin{equation*}
		\sum_{h\in\{1,2\}} \chi_3(h,3) e^{-\frac{2\pi ihn}3}= -2\sin\left(\frac{2\pi n}3\right).
	\end{equation*}
This yields the first identity for $3\nmid n$, and for $n\equiv0\Pmod3$ this vanishes.

For $k=6$, the sum on $h$ becomes
	\[
	\sum_{h\in\{1,5\}} \chi_3(h,6) e^{-\frac{2\pi i hn}{6}}=- 2\sin\left(\frac{\pi n}{3}\right).
	\]
	This also vanishes if $3\mid n$.

Finally, plugging \eqref{E:MultP} into the sum on $h$ for $k=9$ and $3\mid n$ gives
	\[
	\sum_{h\in\{1,2,4,5,7,8\}} \chi_{3}(h,9) e^{-\frac{2\pi i hn}{9}}=\alpha_n.
	\]
This yields the first identity for $3\mid n$ as well.

Next suppose that $|M_3(n)|>|E_3(n)| +|E(n)|+|\sum_{3,2}| + |\sum_1|$. In this case, we have $s_3(n)=\sgn(\alpha_n)$, and one easily checks that
\[
\sgn\left(\alpha_n\right)=\begin{cases}
			1&\text{if }n\equiv 0,2,5,6,8\pmod{9},\\
			-1&\text{if }n\equiv 1,3,4,7\pmod{9},
		\end{cases}
\]
which matches the value claimed in \Cref{T:Pw}.
\end{proof}

	\subsection{Error bounds}
By \Cref{lem:1^93^{-5}almostexact}, we need to bound $|M_3(n)|$ from below and $|E_3(n)|$, $|E(n)|$, $|\sum_{3,2}|$, and $|\sum_1|$ from above. The estimate for $|E(n)|$ is given in \eqref{eqn:Enbound}.
	\subsubsection{Lower bounds for the main term}
Since the definition of $M_3(n)$ depends on $\gcd(n,3)$, we distinguish whether $3\nmid n$ or $3\mid n$.
		If $3\nmid n$, then
		\begin{equation}\label{eqn:main1^93^-5not3}
\left|M_3(n)\right|=\frac{4\pi}{3}\sin\left(\frac{2\pi}3\right) \sqrt{4n-1} I_1\left(\frac\pi{3}\sqrt{4n-1}\right).
		\end{equation}

		If $n\equiv0\Pmod3$, then
		\begin{equation}\label{eqn:main1^93^-53midn}
		\hspace{-.2cm}\left|M_3(n)\right|=\!\frac{2\pi}{9}|\alpha_n|  \sqrt{4n-1} I_1\!\left(\frac\pi{9}\sqrt{4n-1}\right)\!\geq \frac{4\pi}3\sin\left(\frac\pi9\right) \sqrt{4n-1} I_1\left(\frac\pi{9}\sqrt{4n-1}\right)\!.
		\end{equation}

	\subsubsection{Bound of tails from \eqref{eqn:CircleMain}}
Applying the bound $I_1(\frac{\pi}{k}\sqrt{4n-1})\leq I_1(\frac{\pi}{3\ell_n}\sqrt{n-1})$ uniformly for $k\geq 3\ell_n$ with $3\mid k$ and trivially bounding the sum over $h$ in $E_3(n)$ yields the following upper bound for $|E_3(n)|$.
	\begin{lem}\label{lem:E1^93^{-5}Bound}
		We have
		\[
			\left|E_3(n)\right|\leq  \frac{4\pi n}{3} I_1\left(\frac{\pi}{3\ell_n}\sqrt{4n-1}\right).
		\]
	\end{lem}
	\subsubsection{Bounding $\sum_1$}

We require the following bound.
\begin{lem}\label{lem:Pq5Pq39}
Assuming the notation above, we have
\[
\left|\frac{P\left(q_1^\frac{1}{3}\right)^5}{P\left(q_1\right)^9}\right|\leq 72.
\]
\end{lem}
\begin{proof}
 We define
		\begin{equation*}
			\frac{P(q)^5}{P\left(q^3\right)^9} =: \sum_{n\ge0} \beta(n) q^n.
		\end{equation*}
		It is not hard to see that $\beta(n)\ge0$. Thus

		\begin{equation*}
			\left|\frac{P\left(q_1^\frac13\right)^5}{P\left(q_1\right)^9}\right| \le \sum_{n\ge0} \beta(n) |q_1|^\frac n3.
		\end{equation*}
		By \eqref{eqn:Farey}, we have $|q_1|=e^{-\frac{2\pi}{k}\im(\frac{1}{z})}<e^{-\pi}$, and hence

		\begin{equation*}
			\left|\frac{P\left(q_1^\frac13\right)^5}{P\left(q_1\right)^9}\right| \le \sum_{n\ge0} \beta(n) e^{-\frac{\pi n}3} = \frac{P\left(e^{-\frac\pi3}\right)^5}{P\left(e^{-\pi}\right)^9} \le P\left(e^{-\frac\pi3}\right)^5.
		\end{equation*}
		We now use the well-known estimate (for example, see \cite[Theorem 14.5]{A2})
		\begin{equation}\label{E:partbound}
			p(n) < e^{\pi\sqrt{\frac{2n}3}}.
		\end{equation}
		Thus, for $n\ge10$
		\begin{equation*}
			p(n) e^{-\frac{\pi n}3} < e^{\pi\sqrt{\frac{2n}3}-\frac{\pi n}3} \le e^{-\frac{\pi n}{13.5}}.
		\end{equation*}
		Thus
		\begin{equation*}
			P\left(e^{-\frac\pi3}\right) \le 1 + \sum_{n=1}^9 p(n) e^{-\frac{\pi n}3} + \frac{e^{-\frac{10\pi}{13.5}}}{1-e^{-\frac\pi{13.5}}}.
		\end{equation*}
		Explicitly plugging in $p(n)$ for $1\leq n\leq 9$ and raising to the fifth power then yields
		\[
			P\left(e^{-\frac\pi3}\right)^5\leq 72.%
		\]
		This gives the claim.
\end{proof}

Using \eqref{eqn:Farey} and \Cref{lem:Pq5Pq39}, we overall bound (noting that $N+1> \sqrt{n}$)
		\begin{align}
\nonumber			\left|\sum{}_{\substack{\\[6pt]\hspace{-0.05cm}1}}\right| &\le 3^{\frac52}\cdot 72 \sum_{\substack{1\le k<N\\3\nmid k}} k \left(|\vartheta_{h,k}'|+|\vartheta_{h,k}''|\right) \max_{-\vartheta_{h,k}'\leq \Phi\leq \vartheta_{h,k}''} \frac1{|z|^2} e^{-\frac{11\pi}{18k}\re\left(\frac1z\right)-\frac\pi{2k}\re(z)+\frac{2\pi n\re(z)}k}\\
\nonumber			&\le 3^{\frac52}\cdot 72 \sum_{1\le k<N} \frac2{N+1} \left(\frac kn\right)^{-2} e^{-\frac{11\pi}{18\cdot2}+2\pi}
\le 3^\frac52\cdot 144 e^{-\frac{11\pi}{36}+2\pi} n^{\frac{3}{2}}\sum_{1\le k<N} \frac1{k^2}\\
\label{eqn:sum1bound}
&\le 3^\frac52\cdot 144 e^{-\frac{11\pi}{36}+2\pi} \zeta(2)n^{\frac{3}{2}}\leq 757137 n^{\frac{3}{2}}.
		\end{align}

	\subsubsection{Bounding $\sum_{3,2}$}
\begin{lem}\label{lem:fq1-1}
Assuming the notation above, we have
\[
	\left|f_3(q_1)-1\right|\left|q_1\right|^{-\frac{1}{4}} \leq \frac23.
\]
\end{lem}
\begin{proof}
	It is not hard to see that
	\begin{equation*}
		\left|\prod_{n\ge1}\frac{\left(1-q^n\right)^9}{\left(1-q^{3n}\right)^5}-1\right| \le \left|P(|q|)^9P\left(|q|^3\right)^5-1\right|.
	\end{equation*}

	Now we again use $|q_1|<e^{-\pi}$ by \eqref{eqn:Farey}. Using \eqref{E:partbound}, for $n\ge 6$ we have
	\begin{equation*}
		p(n) e^{-\pi n} < e^{\pi\sqrt{\frac{2n}3}-\pi n} \le e^{-\frac{2\pi n}{3}}.
	\end{equation*}
	Thus
	\begin{equation*}
		P\left(e^{-\pi}\right) \le 1 + \sum_{n=1}^5 p(n) e^{-\pi n} + \frac{e^{-4\pi}}{1-e^{-\frac{2\pi}{3}}}.
	\end{equation*}
	Explicitly plugging in $p(n)$ for $1\leq n\leq 5$ then yields
	\[
		1\leq P\left(e^{-\pi}\right)^9\leq 1.52.
	\]
	Similarly, \eqref{E:partbound} gives, for $n\geq 1$,
	\[
		p(n)e^{-3\pi n} < e^{\pi\sqrt{\frac{2n}3}-3\pi n}\le e^{-2\pi n}.
	\]
	Thus
	\[
		1\leq P\left(e^{-3\pi}\right)^5\leq \left(\sum_{n\geq0} e^{-2\pi n}\right)^5 = \frac{1}{\left(1-e^{-2\pi}\right)^5} \leq 1.01.
	\]
	Therefore
	\[
		e^{\frac\pi{24}}\left(P\left(e^{-\pi}\right)^9 P\left(e^{-3\pi}\right)^5-1\right) <\frac{2}{3}.
	\]
	The claim now follows.
\end{proof}
Using \Cref{lem:fq1-1}, we are now able to bound the contribution from $\sideset{}{_{3,2}}{\sum}$.

\begin{lem}\label{lem:sum32}
We have
\[
	\left| \sideset{}{_{3,2}}{\sum}\right|\leq 131n^{\frac{3}{2}}.
\]
\end{lem}
\begin{proof}
		Trivially bounding the sum on $h$ and using \Cref{lem:fq1-1}, we bound
		\begin{multline}\label{eqn:sum32bound1}
			\left| \sideset{}{_{3,2}}{\sum}\right|\leq \sum_{\substack{1\leq k<N\\ 3\mid k}} k \int_{-\vartheta_{h,k}'}^{\vartheta_{h,k}''} \frac{1}{|z|^{2}}  e^{\frac\pi{2k}\left(\re\left(\frac1z\right)-\re(z)\right)}  \left|f_3(q_1)-1\right| e^{\frac{2\pi n\re(z)}k} d\Phi\\
			\leq \frac{2}{3} \sum_{\substack{1\leq k<N\\ 3\mid k}} k\left( \left|\vartheta_{h,k}'\right|+\left|\vartheta_{h,k}''\right|\right)
			\max_{-\vartheta_{h,k}'\leq \Phi\leq \vartheta_{h,k}''}\frac{1}{|z|^{2}}  e^{\frac\pi{2k}\left(\re\left(\frac1z\right)-\re(z)\right)} \left|q_1\right|^{\frac{1}{4}}    e^{\frac{2\pi n\re(z)}k}.
\end{multline}
By \eqref{eqn:Farey}, we have $|z|^{-2}<(\frac{k}{n})^{-2}$ and $k|\vartheta_{h,k}'|\le\frac1{N+1}$.
Moreover, we have
\[
e^{\frac\pi{2k}\re\left(\frac1z\right)}|q_1|^{\frac{1}{4}}=1.
\]
Plugging back into \eqref{eqn:sum32bound1} and using $\re(z)=\frac{k}{n}$ (by \eqref{eqn:Farey}) then yields
\begin{equation*}
\left| \sideset{}{_{3,2}}{\sum}\right| \leq \frac{4n^2}{3(N+1)}e^{2\pi}\sum_{\substack{1\leq k<N\\ 3\mid k}} \frac{1}{k^2} \leq \frac{4}{27} e^{2\pi} \zeta(2) n^{\frac{3}{2}}= \frac{2\pi^2 }{81} e^{2\pi}n^{\frac{3}{2}} \leq 131 n^{\frac{3}{2}}.\qedhere
\end{equation*}
\end{proof}
	\subsection{Finishing the proof}
	\subsubsection{$3\nmid n$}
Assume that $3\nmid n$. Plugging \eqref{eqn:Enbound}, \eqref{eqn:main1^93^-5not3}, \Cref{lem:E1^93^{-5}Bound}, \eqref{eqn:sum1bound}, and \Cref{lem:sum32} into \Cref{lem:1^93^{-5}almostexact}, we conclude that $s_3(n)$ is as claimed in \Cref{T:Pw} if
	\begin{multline}\label{eqn:tofinishnot3}
		\left(\frac{4\pi}3\sin\left(\frac{2\pi}3\right)\sqrt{4n-1} I_1\left(\frac\pi3\sqrt{4n-1}\right)\right)^{-1}\\
		\times \left(\frac{4\pi n}{3} I_1\left(\frac\pi6\sqrt{4n-1}\right)+
757137 n^{\frac{3}{2}}+131 n^{\frac{3}{2}}+\frac{2294012\pi^2}{27}n^\frac32\right)<1.
\end{multline}
We first estimate the left-hand side of \eqref{eqn:tofinishnot3} against
\begin{equation}\label{eqn:tofinishnot3-2}
\hspace{-.1cm}\!\left(\frac{4\pi}3\sin\left(\frac{2\pi}3\right)\sqrt{4n-1} I_1\!\left(\frac\pi3\sqrt{4n-1}\right)\!\right)^{-1}\!\left(\frac{4\pi n}{3}I_1\left(\frac\pi6\sqrt{4n-1}\right)+1595824n^{\frac{3}{2}}\right)\!.
\end{equation}
For $n\geq 6$, we may use \Cref{L:Bessel} (2) and (3) to bound \eqref{eqn:tofinishnot3-2} from above by
\[
\frac{\sqrt3e^{-\frac\pi3\sqrt{4n-1}}}{\sqrt\pi\sin\left(\frac{2\pi}3\right)(4n-1)^\frac14}  \left(\frac{8 n  e^{\frac\pi6\sqrt{4n-1}}}{\sqrt{3}(4n-1)^\frac14}+1595824n^\frac32\right).
\]
One may check that this is less than $1$ for $n\geq 89$, and for $n<89$ we directly evaluate $s_3(n)$ with a computer, verifying \Cref{T:Pw} for $3\nmid n$.

	\subsubsection{$3\mid n$}
Next assume that $3\mid n$. Plugging \eqref{eqn:Enbound}, \eqref{eqn:main1^93^-53midn}, \Cref{lem:E1^93^{-5}Bound}, \eqref{eqn:sum1bound}, and \Cref{lem:sum32} into \Cref{lem:1^93^{-5}almostexact}, we conclude that $s_3(n)$ agrees with the value in the claim of \Cref{T:Pw} if
\begin{multline}\label{eqn:tofinish3midn}
		\left(\frac{4\pi}3\sin\left(\frac\pi9\right)\sqrt{4n-1} I_1\left(\frac\pi9\sqrt{4n-1}\right)\right)^{-1}\\
		\times \left(\frac{4\pi n}{3}I_1\left(\frac\pi{12}\sqrt{4n-1}\right)+757137n^{\frac{3}{2}}+131 n^{\frac{3}{2}}+\frac{2294012\pi^2}{27}n^\frac32\right)<1.
\end{multline}
We first bound the left-hand side of \eqref{eqn:tofinish3midn} from above by
\begin{equation}\label{eqn:tofinish3midn-2}
\left(\frac{4\pi}3\sin\left(\frac\pi9\right)\sqrt{4n-1} I_1\left(\frac\pi9\sqrt{4n-1}\right)\right)^{-1}\left(\frac{4\pi n}{3}I_1\left(\frac\pi{12}\sqrt{4n-1}\right)+1595824n^{\frac{3}{2}}\right).
\end{equation}
For $n\geq 19$, \eqref{eqn:tofinish3midn-2} may be bounded from above by
\[
\frac1{\sqrt\pi\sin\left(\frac\pi9\right)(4n-1)^\frac14} e^{-\frac\pi9\sqrt{4n-1}} \left(\frac{8\sqrt{2} n e^{\frac\pi{12}\sqrt{4n-1}}}{\sqrt{3}(4n-1)^\frac14} +1595824n^\frac32\right).
\]
 One can check that this is less than $1$ for $n\geq 1173$. Confirming \Cref{T:Pw} for $n<1173$ directly with a computer, we conclude \Cref{T:Pw}.

\appendix
\section{Further comptuational evidence for purely periodic sign changes}\label{sec:Conjectures}
For each sign pattern listed below, we list further choices of $1^{\delta_1}2^{\delta_2}\cdots m^{\delta_m}$ for which computational data indicates that $s_{1^{\delta_1}2^{\delta_2}\cdots m^{\delta_m}}(n)$ satisfies the given purely periodic sign pattern.
{\footnotesize\setlength{\tabcolsep}{2pt}
\begin{longtable}{|c|c|c|c|c|c|c|c|}
\caption{The list of conjectures.}
\\
\hline%
\tiny period/&\multicolumn{5}{|c|}{\multirow{2}{*}{case}}\\
\tiny sign pattern&\multicolumn{5}{|c|}{}\\
\hline
\multirow{2}{*}{$\begin{smallmatrix}2/\\ +-\end{smallmatrix}$}
 &$1^{2}2^{-3}3^{-3}$\vphantom{$x^{x^{x^x}}$}&
$1^{2}2^{-2}3^{4}$&
$1^{3}2^{-2}$&
$1^{3}2^{-2}3^{1}$&
$1^{3}2^{-2}3^{2}$\\
&$1^{3}2^{-2}3^{3}$
&$1^{4}2^{-4}3^{-4}$&
$1^{4}2^{-3}3^{4}$& &
\\[2pt]
\hline
\multirow{2}{*}{$\begin{smallmatrix}3/\\++-\end{smallmatrix}$}
&$1^{-1}2^{2}3^{-1}$\vphantom{$x^{x^{x^x}}$}&
$1^{-1}2^{3}3^{-3}4^{1}$&
$1^{-1}2^{3}3^{-2}$&
$1^{-1}2^{4}3^{-4}4^{2}$&
$1^{-1}2^{4}3^{-3}4^{1}$\\
&$1^{-1}2^{4}3^{-2}$&
$1^{-1}2^{5}3^{-4}4^{3}$&
$1^{-1}2^{5}3^{-3}$&
 $1^{-1}2^{5}3^{-3}4^{1}$&
$1^{-1}2^{5}3^{-3}4^{2}$
\\[2pt]
\hline
\multirow{2}{*}{$\begin{smallmatrix}3\text{\vphantom{$x^{x^x}$}}/\\+0+\end{smallmatrix}$}
 & $1^{3}2^{-2}3^{-3}$\vphantom{$x^{x^{x^x}}$}&
$1^{3}2^{-1}3^{-2}$&
$1^{4}2^{-2}3^{-3}$&
$1^{4}2^{-1}3^{-2}$&
$1^{4}3^{-3}$\\
&$2^{-1}3^{-1}4^{2}$&
$1^{4}3^{-3}4^{4}$ &
$1^{4}3^{-2}4^{1}$&
$1^{4}3^{-2}4^{2}$&
$1^{4}3^{-2}4^{3}$
\\[2pt]
\hline
$\begin{smallmatrix}3/\\+0-\end{smallmatrix}$& $2^{2}3^{-1}4^{-1}$\vphantom{$x^{x^{x^x}}$}&
$2^{3}3^{-1}$&
$1^{1}2^{-1}3^{-2}4^{1}$&
$1^{2}2^{-1}3^{-2}$&
$1^{3}3^{-2}$\\[2pt]
\hline
\multirow{3}{*}{$\begin{smallmatrix}3/ \\ +--\end{smallmatrix}$}
&\vphantom{$x^{x^{x^x}}$}$1^{1}2^{1}3^{-2}$&
$1^{1}2^{2}3^{-2}$&
$1^{1}2^{3}3^{-4}$&
$1^{2}2^{1}3^{-2}$&
$1^{2}2^{2}3^{-2}$\\
&$1^{2}2^{3}3^{-3}$&
$1^{2}2^{4}3^{-4}$&
$1^{3}2^{1}3^{-2}$&
$1^{3}2^{2}3^{-2}$&
$1^{3}2^{3}3^{-2}$\\
&$1^{3}2^{4}3^{-2}$&
$1^{4}2^{3}3^{-2}$&
$1^{4}2^{4}3^{-2}$&&
\\[2pt]
\hline
\multirow{2}{*}{$\begin{smallmatrix}4/\\  +++0\end{smallmatrix}$}&
$1^{-4}2^{10}4^{-5}$\vphantom{$x^{x^{x^x}}$}&
$1^{-3}2^{8}4^{-5}5^{-1}$&
$1^{-3}2^{8}3^{1}4^{-7}$&
$1^{-1}3^{3}4^{-1}$&
$1^{-1}3^{4}4^{-3}5^{1}$
\\
&$1^{-1}2^{1}3^{3}4^{-4}5^{-2}$&
$1^{-1}2^{1}3^{3}4^{-3}$&&&
\\[2pt]
\hline
$\begin{smallmatrix}4/\\++--\end{smallmatrix}$ &
$1^{-1}2^{3}3^{1}4^{-3}$\vphantom{$x^{x^{x^x}}$}&
$1^{-1}2^{3}3^{2}4^{-4}$&
$1^{-1}2^{4}4^{-4}$&
$1^{-1}2^{4}3^{1}4^{-3}$&
$1^{-1}2^{4}3^{2}4^{-4}$
\\[2pt]
\hline
$\begin{smallmatrix}4/\\ ++0+\end{smallmatrix}$ & $1^{-1}2^{2}4^{-2}5^{1}$\vphantom{$x^{x^{x^x}}$}&&&&
\\[2pt]
\hline
\multirow{5}{*}{$\begin{smallmatrix}4/ \\ +--+\end{smallmatrix}$} &
$1^{1}3^{-1}4^{-3}$\vphantom{$x^{x^{x^x}}$}&
$1^{1}2^{1}3^{-1}4^{-3}$&
$1^{1}2^{1}4^{-3}$&
$1^{1}2^{2}3^{-1}4^{-3}$&
$1^{1}2^{2}4^{-3}$\\
&$1^{1}2^{3}3^{-1}4^{-3}$&
$1^{1}2^{3}4^{-3}$&
$1^{1}2^{4}3^{-1}4^{-4}$&
$1^{1}2^{4}4^{-3}$&
$1^{2}3^{-1}4^{-3}$\\
&$1^{2}4^{-2}$&
$1^{2}3^{1}4^{-2}$&
$1^{2}2^{1}3^{-1}4^{-3}$&
$1^{2}2^{1}4^{-2}$&
$1^{2}2^{2}3^{-1}4^{-3}$\\
&$1^{2}2^{2}4^{-2}$&
$1^{2}2^{3}3^{-1}4^{-3}$&
$1^{2}2^{3}4^{-3}$&
$1^{2}2^{4}3^{-1}4^{-4}$&
$1^{2}2^{4}4^{-3}$\\
&$1^{3}2^{1}4^{-2}$&
$1^{3}2^{2}4^{-2}$&
$1^{3}2^{3}4^{-3}$&
$1^{3}2^{4}4^{-3}$&
\\[2pt]
\hline
$\begin{smallmatrix}4/\\+---\end{smallmatrix}$& $1^{1}3^{1}4^{-3}$\vphantom{$x^{x^{x^x}}$}&&&&
\\[2pt]
\hline
$\begin{smallmatrix}4/\\+-00\end{smallmatrix}$ &
$1^{2}2^{-1}4^{-1}$\vphantom{$x^{x^{x^x}}$}&&&&
\\[2pt]
\hline
$\begin{smallmatrix}4/\\+--0\end{smallmatrix}$ & $1^{2}3^{2}4^{-3}$\vphantom{$x^{x^{x^x}}$}&&&&
\\[2pt]
\hline
$\begin{smallmatrix}4/\\+-++\end{smallmatrix}$ &
$1^{3}2^{-1}3^{-1}4^{-2}$\vphantom{$x^{x^{x^x}}$}&
$1^{3}2^{-1}4^{-2}5^{1}$&
$1^{4}2^{-1}4^{-3}$&
$1^{4}4^{-6}$&
$1^{4}2^{1}4^{-9}$
\\[2pt]
\hline
$\begin{smallmatrix}4/\\+-0+\end{smallmatrix}$ & $1^{3}3^{-1}4^{-4}$\vphantom{$x^{x^{x^x}}$}&&&&
\\[2pt]
\hline
$\begin{smallmatrix}4/\\+-+0\end{smallmatrix}$ & $1^{4}2^{-2}4^{-1}$\vphantom{$x^{x^{x^x}}$}&&&&
\\[2pt]
\hline
\multirow{2}{*}{$\begin{smallmatrix}5/\\ +++--\end{smallmatrix}$} &
$1^{-2}2^{3}3^{5}5^{-5}$\vphantom{$x^{x^{x^x}}$}&
$1^{-2}2^{3}3^{6}5^{-7}$&
$1^{-2}2^{4}3^{3}5^{-5}$&
$1^{-2}2^{4}3^{4}5^{-5}$&
$1^{-2}2^{4}3^{5}5^{-6}$\\[2pt]
&$1^{-2}2^{4}3^{6}5^{-7}$&
$1^{-2}2^{4}3^{7}5^{-8}$&
$1^{-1}2^{1}3^{3}5^{-3}$&&
\\[2pt]
\hline
\multirow{1}{*}{$\begin{smallmatrix}5/\\++++-\end{smallmatrix}$} & $1^{-3}2^{5}3^{4}5^{-9}$\vphantom{$x^{x^{x^x}}$}&
$1^{-2}2^{3}3^{2}4^{2}5^{-4}$&
$1^{-2}2^{3}3^{3}4^{1}5^{-4}$&
$1^{-1}2^{1}4^{3}5^{-2}$&
$1^{-1}2^{1}3^{1}4^{2}5^{-2}$\\[2pt]
\hline
$\begin{smallmatrix}5/\\+++-+\end{smallmatrix}$ &
$1^{-3}2^{8}4^{-4}5^{-4}$\vphantom{$x^{x^{x^x}}$}&
$1^{-1}3^{4}5^{-2}$&
$1^{-1}2^{1}3^{3}4^{-2}5^{-5}$&&
\\[2pt]
\hline
$\begin{smallmatrix}5/\\++00+\end{smallmatrix}$ & $1^{-1}2^{4}4^{-2}5^{-3}$\vphantom{$x^{x^{x^x}}$}&&&&
\\[2pt]
\hline
\multirow{4}{*}{$\begin{smallmatrix}5/\\ ++---\end{smallmatrix}$} &
$1^{-1}2^{3}3^{1}5^{-3}$\vphantom{$x^{x^{x^x}}$}&
$1^{-1}2^{3}3^{2}5^{-3}$&
$1^{-1}2^{3}3^{3}5^{-3}$&
$1^{-1}2^{4}5^{-4}$&
$1^{-1}2^{4}3^{1}5^{-3}$\\
&$1^{-1}2^{4}3^{2}5^{-3}$&
$1^{-1}2^{4}3^{3}5^{-3}$&
$1^{-1}2^{4}3^{4}5^{-4}$&
$1^{-1}2^{5}3^{1}5^{-3}$&
$1^{-1}2^{5}3^{2}5^{-2}$\\
&$1^{-1}2^{5}3^{3}5^{-3}$&
$1^{-1}2^{5}3^{4}5^{-3}$&
$1^{-1}2^{5}3^{5}5^{-4}$&
$1^{-1}2^{6}3^{3}5^{-3}$&
$1^{-1}2^{6}3^{4}5^{-3}$\\
&$1^{-1}2^{6}3^{5}5^{-2}$&
$1^{-1}2^{6}3^{6}5^{-3}$&
$1^{-1}2^{7}3^{6}5^{-3}$&
$1^{-1}2^{7}3^{7}5^{-5}$&
\\[2pt]
\hline
$\begin{smallmatrix}5/\\++0+0\end{smallmatrix}$ &
$1^{-1}2^{2}5^{-1}$\vphantom{$x^{x^{x^x}}$}&&&&
\\[2pt]
\hline
$\begin{smallmatrix}5/\\++-00\end{smallmatrix}$ &  $1^{-1}2^{3}4^{-1}5^{-2}$\vphantom{$x^{x^{x^x}}$}&&&&
\\[2pt]
\hline
$\begin{smallmatrix}5/\\+0-0-\end{smallmatrix}$ &
$2^{1}5^{-1}$\vphantom{$x^{x^{x^x}}$}&
$1^{1}2^{-1}4^{1}5^{-2}$&&&
\\[2pt]
\hline
$\begin{smallmatrix}5/\\+-+--\end{smallmatrix}$ & $1^{1}2^{-2}4^{4}5^{-3}$\vphantom{$x^{x^{x^x}}$}&&&&
\\[2pt]
\hline
\multirow{4}{*}{$\begin{smallmatrix}5/\\ +-+-+\end{smallmatrix}$} &
$1^{1}2^{-2}5^{-8}$\vphantom{$x^{x^{x^x}}$}&
$1^{1}2^{-2}4^{1}5^{-6}$&
$1^{1}2^{-2}4^{2}5^{-5}$&
$1^{2}2^{-3}4^{2}5^{-7}$&
$1^{2}2^{-2}5^{-7}$\\
&$1^{2}2^{-2}4^{1}5^{-5}$&
$1^{2}2^{-2}4^{2}5^{-3}$&
$1^{2}2^{-2}4^{3}5^{-3}$&
$1^{3}2^{-3}4^{3}5^{-4}$&
$1^{3}2^{-2}5^{-7}$\\
&$1^{3}2^{-2}4^{1}5^{-3}$&
$1^{3}2^{-2}4^{2}5^{-3}$&
$1^{3}2^{-2}4^{3}5^{-2}$&
$1^{3}2^{-2}4^{4}5^{-2}$&
$1^{4}2^{-3}4^{4}5^{-3}$\\
&$1^{5}2^{-3}4^{5}5^{-3}$&
$1^{10}2^{-3}5^{-5}$&&&
\\[2pt]
\hline
\multirow{3}{*}{$\begin{smallmatrix}5/\\ +--+-\end{smallmatrix}$} &
$1^{1}3^{-2}5^{-6}$\vphantom{$x^{x^{x^x}}$}&
$1^{1}3^{-1}5^{-4}$&
$1^{1}3^{-1}4^{1}5^{-4}$&
$1^{1}2^{1}3^{-1}5^{-5}$&
$1^{2}3^{-2}4^{1}5^{-6}$\\
&$1^{2}3^{-1}5^{-4}$&
$1^{2}3^{-1}4^{1}5^{-3}$&
$1^{2}3^{-1}4^{2}5^{-3}$&
$1^{3}2^{1}3^{-1}5^{-4}$&
$1^{3}2^{1}5^{-3}$\\
&$1^{3}2^{2}5^{-8}$&&&&
\\[2pt]
\hline
\multirow{2}{*}{$\begin{smallmatrix}5/\\ +---+\end{smallmatrix}$} &
$1^{1}3^{1}4^{-1}5^{-4}$\vphantom{$x^{x^{x^x}}$}&
$1^{1}3^{1}5^{-3}$&
$1^{1}3^{2}5^{-3}$&
$1^{1}3^{3}4^{-1}5^{-3}$&
$1^{1}3^{3}5^{-2}$\\
&$1^{1}3^{4}5^{-3}$&
$1^{2}3^{3}5^{-3}$&&&
\\
\hline
\multirow{3}{*}{$\begin{smallmatrix}5/\\ +--++\end{smallmatrix}$} &
$1^{1}2^{1}4^{-1}5^{-3}$\vphantom{$x^{x^{x^x}}$}&
$1^{1}2^{2}4^{-2}5^{-3}$&
$1^{1}2^{2}4^{-1}5^{-3}$&
$1^{1}2^{2}5^{-5}$&
$1^{1}2^{3}4^{-1}5^{-3}$\\
&$1^{2}4^{-1}5^{-4}$&
$1^{2}5^{-3}$&
$1^{2}3^{1}5^{-4}$&
$1^{2}2^{1}5^{-5}$&
$1^{2}2^{2}4^{-1}5^{-3}$\\
&$1^{2}2^{2}5^{-2}$&
$1^{2}2^{3}5^{-5}$&&&
\\[2pt]
\hline
$\begin{smallmatrix}5/\\+-00+\end{smallmatrix}$ &
$1^{2}2^{-1}5^{-2}$\vphantom{$x^{x^{x^x}}$}&&&&
\\[2pt]
\hline
$\begin{smallmatrix}5/\\+-++-\end{smallmatrix}$ &
$1^{4}3^{-1}5^{-3}$\vphantom{$x^{x^{x^x}}$}&
$1^{4}4^{1}5^{-3}$&&&
\\[2pt]
\hline
$\begin{smallmatrix}5/\\+-+++\end{smallmatrix}$ &
$1^{3}2^{-1}5^{-2}$\vphantom{$x^{x^{x^x}}$}&
$1^{3}2^{-1}4^{1}5^{-3}$&
$1^{3}2^{-1}3^{1}5^{-3}$&
$1^{4}2^{-1}3^{1}4^{1}5^{-4}$&
$1^{4}3^{2}4^{1}5^{-5}$
\\[2pt]
\hline
$\begin{smallmatrix}5/\\+-0+0\end{smallmatrix}$ & $1^{3}5^{-1}$\vphantom{$x^{x^{x^x}}$}&&&&
\\[2pt]
\hline
\multirow{4}{*}{$\begin{smallmatrix}6/\\ +-+-++\end{smallmatrix}$} &
$1^{1}2^{-8}3^{-6}4^{7}5^{5}$\vphantom{$x^{x^{x^x}}$}&
$1^{1}2^{-8}3^{-6}4^{8}5^{5}$&
$1^{1}2^{-7}3^{-6}4^{6}5^{5}$&
$1^{2}2^{-6}3^{-8}4^{4}$&
$1^{2}2^{-6}3^{-7}4^{4}$\\
&$1^{3}2^{-7}3^{-10}4^{1}5^{1}$&
$1^{5}2^{-6}3^{-8}5^{-1}$&
$1^{5}2^{-6}3^{-7}4^{1}$&
$1^{6}2^{-7}3^{-10}4^{-1}$&
$1^{6}2^{-7}3^{-9}5^{1}$\\
&$1^{6}2^{-6}3^{-8}4^{-1}5^{2}$&
$1^{6}2^{-4}3^{-3}4^{1}5^{-1}$&
$1^{7}2^{-7}3^{-10}4^{-2}5^{3}$&
$1^{7}2^{-6}3^{-6}4^{1}5^{-1}$&
$1^{7}2^{-5}3^{-5}$\\
&$1^{8}2^{-6}3^{-7}4^{-1}5^{1}$&&&&
\\[2pt]
\hline
\multirow{2}{*}{$\begin{smallmatrix}6/\\ +-+++-\end{smallmatrix}$} &
$1^{1}2^{-5}3^{-6}4^{-3}5^{5}$\vphantom{$x^{x^{x^x}}$}&
$1^{1}2^{-5}3^{-6}4^{-2}5^{5}$&
$1^{1}2^{-5}3^{-6}4^{-1}5^{5}$&
$1^{1}2^{-4}3^{-6}4^{-3}5^{5}$&
$1^{1}2^{-4}3^{-6}4^{-2}5^{5}$\\ 
&$1^{1}2^{-4}3^{-6}4^{-1}5^{5}$&
$1^{1}2^{-3}3^{-6}4^{-2}5^{5}$&&&
\\[2pt]
\hline
$\begin{smallmatrix}6/\\ +-++++\end{smallmatrix}$ &
$1^{1}2^{-5}3^{-6}4^{2}5^{5}$\vphantom{$x^{x^{x^x}}$}&$1^{1}2^{-4}3^{-5}4^{4}$&&&
\\[2pt]
\hline
$\begin{smallmatrix}6/\\+-+-+0\end{smallmatrix}$ &
$1^{1}2^{-4}3^{-3}4^{4}$\vphantom{$x^{x^{x^x}}$}&
$1^{2}2^{-5}3^{-6}4^{2}$&
$1^{3}2^{-3}3^{-1}4^{3}$&
$1^{8}2^{-7}3^{-8}$&
$1^{9}2^{-5}3^{-3}4^{1}$
\\[2pt]
\hline
$\begin{smallmatrix}6/\\+-+0+0\end{smallmatrix}$ &
$1^{1}2^{-3}3^{-3}4^{2}$\vphantom{$x^{x^{x^x}}$}&&&&
\\[2pt]
\hline
$\begin{smallmatrix}6/\\ +-+0-+\end{smallmatrix}$ &
$1^{1}2^{-3}3^{-3}4^{5}$\vphantom{$x^{x^{x^x}}$}&
$1^{2}2^{-4}3^{-6}4^{3}$&
$1^{4}2^{-3}3^{-4}4^{2}$&
$1^{8}2^{-1}3^{-8}$&
\\[2pt]
\hline
$\begin{smallmatrix}6/\\ +-+--+\end{smallmatrix}$ &
$1^{2}2^{-3}3^{-3}4^{4}$\vphantom{$x^{x^{x^x}}$}&
$1^{2}2^{-3}3^{-2}4^{4}$&
$1^{3}2^{-3}3^{-3}4^{3}$&
$1^{4}2^{-3}3^{-3}4^{2}$&
\\[2pt]
\hline
$\begin{smallmatrix}6/\\ +-++0-\end{smallmatrix}$ & $1^{1}2^{-2}3^{-3}$\vphantom{$x^{x^{x^x}}$}&&&&
\\[2pt]
\hline
\multirow{3}{*}{$\begin{smallmatrix}6/\\ +-++--\end{smallmatrix}$} &
$1^{2}2^{-3}3^{-5}5^{2}$\vphantom{$x^{x^{x^x}}$}&
$1^{2}2^{-2}3^{-4}5^{1}$&
$1^{2}2^{-2}3^{-3}5^{1}$&
$1^{3}2^{-2}3^{-3}5^{2}$&
$1^{3}2^{-2}3^{-2}5^{2}$\\
&$1^{4}2^{1}3^{-6}$&
$1^{4}2^{1}3^{-5}$&
$1^{4}2^{1}3^{-4}$&
$1^{5}2^{1}3^{-3}4^{-1}$&
$1^{5}2^{2}3^{-6}$\\
&$1^{5}2^{2}3^{-5}$&
$1^{5}2^{2}3^{-4}$&
$1^{5}2^{2}3^{-3}$&&
\\[2pt]
\hline
$\begin{smallmatrix}6/\\+-+00-\end{smallmatrix}$ &
$1^{1}2^{-2}3^{-2}4^{1}5^{1}$\vphantom{$x^{x^{x^x}}$}&&&&
\\[2pt]
\hline
$\begin{smallmatrix}6/\\+-0+0-\end{smallmatrix}$ &
$1^{1}2^{-1}3^{-3}4^{-2}$\vphantom{$x^{x^{x^x}}$}&
$1^{1}2^{-1}3^{-2}4^{-1}5^{1}$&&&
\\[2pt]
\hline
$\begin{smallmatrix}6/\\ +--+-0\end{smallmatrix}$ &
$1^{1}3^{-3}4^{2}$\vphantom{$x^{x^{x^x}}$}&&&&
\\[2pt]
\hline
$\begin{smallmatrix}6/\\+--++-\end{smallmatrix}$ &
$1^{1}2^{4}3^{-5}$\vphantom{$x^{x^{x^x}}$}&
$1^{1}2^{4}3^{-4}$&&&
\\[2pt]
\hline
$\begin{smallmatrix}6/\\+-+0-0\end{smallmatrix}$ &
$1^{2}2^{-4}3^{-6}$\vphantom{$x^{x^{x^x}}$}&
$1^{4}2^{-3}3^{-4}4^{-1}$&&&
\\[2pt]
\hline
$\begin{smallmatrix}6/\\+-+0--\end{smallmatrix}$ &
$1^{2}2^{-3}3^{-4}5^{2}$\vphantom{$x^{x^{x^x}}$}&&&&
\\[2pt]
\hline
$\begin{smallmatrix}6/\\+-++-0\end{smallmatrix}$ &
$1^{2}2^{-2}3^{-6}4^{-1}$\vphantom{$x^{x^{x^x}}$}&
$1^{4}3^{-4}4^{-1}$&
$1^{5}2^{2}3^{-7}$&&
\\[2pt]
\hline
$\begin{smallmatrix}6/\\+--+-+\end{smallmatrix}$ &
$1^{2}2^{1}3^{-2}4^{2}$\vphantom{$x^{x^{x^x}}$}&&&&
\\[2pt]
\hline
$\begin{smallmatrix}6/\\+-+-0+\end{smallmatrix}$&
$1^{3}2^{-6}3^{-9}$&
$1^{5}2^{-5}3^{-6}5^{1}$\vphantom{$x^{x^{x^x}}$}&&&
\\[2pt]
\hline
$\begin{smallmatrix}6/\\+-+---\end{smallmatrix}$ &
$1^{3}2^{-5}3^{-7}5^{2}$&
$1^{3}2^{-4}3^{-6}4^{-1}5^{3}$&
$1^{4}2^{-3}3^{-3}5^{1}$\vphantom{$x^{x^{x^x}}$}&&
\\[2pt]
\hline
$\begin{smallmatrix}6/\\+-+-00\end{smallmatrix}$ &
$1^{4}2^{-4}3^{-4}4^{1}$\vphantom{$x^{x^{x^x}}$}&&&&
\\[2pt]
\hline
$\begin{smallmatrix}6/\\+-0+--\end{smallmatrix}$ &
$1^{4}2^{2}3^{-4}4^{1}$\vphantom{$x^{x^{x^x}}$}&&&&
\\[2pt]
\hline
$\begin{smallmatrix}6/\\+-+--0\end{smallmatrix}$ &
$1^{6}2^{-3}3^{-2}$\vphantom{$x^{x^{x^x}}$}&&&&
\\[2pt]
\hline
$\begin{smallmatrix}8/\\ +--+++--\end{smallmatrix}$ &
$1^{1}2^{4}3^{1}4^{-3}$\vphantom{$x^{x^{x^x}}$}&&&&
\\[2pt]
\hline
$\begin{smallmatrix}8/\\ +-++0-0+\end{smallmatrix}$ &
$1^{4}4^{-5}$\vphantom{$x^{x^{x^x}}$}&&&&
\\[2pt]
\hline
\multirow{2}{*}{$\begin{smallmatrix}8/\\ +-++---+\end{smallmatrix}$} &
$1^{4}4^{-4}$\vphantom{$x^{x^{x^x}}$}&
$1^{4}4^{-3}$&
$1^{4}4^{-2}$&
$1^{4}2^{1}4^{-6}$&
$1^{4}2^{1}4^{-5}$\\
&$1^{4}2^{1}4^{-4}$&
$1^{4}2^{1}4^{-3}$&&&
\\[2pt]
\hline
$\begin{smallmatrix}8/\\+-0+0-0+\end{smallmatrix}$ &
$1^{4}2^{2}4^{-10}$\vphantom{$x^{x^{x^x}}$}&&&&
\\[2pt]
\hline
\multirow{2}{*}{$\begin{smallmatrix}8/\\ +-0+--0+\end{smallmatrix}$} &
$1^{4}2^{2}4^{-9}$\vphantom{$x^{x^{x^x}}$}&
$1^{4}2^{2}4^{-8}$&
$1^{4}2^{2}4^{-7}$&
$1^{4}2^{2}4^{-6}$&
$1^{4}2^{2}4^{-5}$\\
&$1^{4}2^{2}4^{-4}$&
$1^{4}2^{2}4^{-3}$&
$1^{4}2^{2}4^{-2}$&&
\\[2pt]
\hline
$\begin{smallmatrix}8/\\+--+--++\end{smallmatrix}$ &
$1^{4}2^{3}4^{-4}$\vphantom{$x^{x^{x^x}}$}&
$1^{4}2^{3}4^{-3}$&
$1^{4}2^{4}4^{-4}$&
$1^{4}2^{4}4^{-3}$&
\\[2pt]
\hline
$\begin{smallmatrix}9/\\+-+--+--+\end{smallmatrix}$ &
$1^{9}3^{-10}$\vphantom{$x^{x^{x^x}}$}&&&&
\\[2pt]
\hline
$\begin{smallmatrix}9/\\+-+--+0-+\end{smallmatrix}$ &
$1^{9}3^{-9}$\vphantom{$x^{x^{x^x}}$}&&&&
\\[2pt]
\hline
$\begin{smallmatrix}9/\\+-+--++-+\end{smallmatrix}$ &
$1^{9}3^{-8}$\vphantom{$x^{x^{x^x}}$}&
$1^{9}3^{-7}$&
$1^{9}3^{-6}$&
$1^{9}3^{-4}$&
\\[2pt]
\hline
$\begin{smallmatrix}10/\\++-0-0+---\end{smallmatrix}$ &
$1^{-1}2^{3}4^{1}5^{-3}$\vphantom{$x^{x^{x^x}}$}&&&&
\\[2pt]
\hline
$\begin{smallmatrix}10/\\+0+-+-+-+0\end{smallmatrix}$ &
$2^{-2}3^{3}5^{-5}$\vphantom{$x^{x^{x^x}}$}&&&&
\\[2pt]
\hline
$\begin{smallmatrix}10/\\+-+-++0-0+\end{smallmatrix}$ &
$1^{1}2^{-2}5^{-5}$\vphantom{$x^{x^{x^x}}$}&&&&
\\[2pt]
\hline
$\begin{smallmatrix}10/\\+-0-++--0+\end{smallmatrix}$ &
$1^{1}2^{-1}4^{-1}5^{-5}$\vphantom{$x^{x^{x^x}}$}&&&&
\\[2pt]
\hline
$\begin{smallmatrix}10/\\+-+-+-+-++\end{smallmatrix}$ &
$1^{2}2^{-4}5^{-10}$\vphantom{$x^{x^{x^x}}$}&
$1^{3}2^{-4}4^{1}5^{-7}$&&&
\\[2pt]
\hline
$\begin{smallmatrix}10/\\+-+-+0-0-+\end{smallmatrix}$ &
$1^{2}2^{-3}5^{-10}$\vphantom{$x^{x^{x^x}}$}&&&&
\\[2pt]
\hline
$\begin{smallmatrix}10/\\+-+-+----+\end{smallmatrix}$ &
$1^{2}2^{-3}5^{-9}$\vphantom{$x^{x^{x^x}}$}&
$1^{2}2^{-3}5^{-8}$&&&
\\[2pt]
\hline
$\begin{smallmatrix}10/\\+-+-++--++\end{smallmatrix}$ &
$1^{2}2^{-2}3^{-1}5^{-5}$\vphantom{$x^{x^{x^x}}$}&&&&
\\[2pt]
\hline
$\begin{smallmatrix}10/\\+-+-++---+\end{smallmatrix}$ &
$1^{2}2^{-2}4^{1}5^{-3}$\vphantom{$x^{x^{x^x}}$}&&&&
\\[2pt]
\hline
$\begin{smallmatrix}10/\\+-+-+00-00\end{smallmatrix}$ &
$1^{2}2^{-2}4^{1}5^{-2}$\vphantom{$x^{x^{x^x}}$}&&&&
\\[2pt]
\hline
$\begin{smallmatrix}10/\\+--+0+--0-\end{smallmatrix}$ &
$1^{2}4^{1}5^{-2}$\vphantom{$x^{x^{x^x}}$}&&&&
\\[2pt]
\hline
$\begin{smallmatrix}10/\\+--+++--+-\end{smallmatrix}$ &
$1^{2}2^{1}5^{-3}$&
$1^{2}2^{1}5^{-2}$\vphantom{$x^{x^{x^x}}$}&&&
\\[2pt]
\hline
$\begin{smallmatrix}10/\\+--++0--+0\end{smallmatrix}$ &
$1^{2}2^{2}4^{-1}5^{-2}$\vphantom{$x^{x^{x^x}}$}&&&&
\\[2pt]
\hline
\multirow{2}{*}{$\begin{smallmatrix}10/\\+-+-+--+++\end{smallmatrix}$} &
$1^{3}2^{-4}3^{-2}5^{-10}$\vphantom{$x^{x^{x^x}}$}&
$1^{3}2^{-3}3^{-1}5^{-7}$&
$1^{3}2^{-2}5^{-4}$&
$1^{3}2^{-2}5^{-3}$&
\\
&$1^{4}2^{-3}3^{-1}4^{1}5^{-5}$&$1^{4}2^{-3}3^{-1}4^{1}5^{-4}$&&&
\\[2pt]
\hline
$\begin{smallmatrix}10/\\+--+-0--+0\end{smallmatrix}$ &
$1^{3}2^{2}5^{-7}$\vphantom{$x^{x^{x^x}}$}&&&&
\\[2pt]
\hline
$\begin{smallmatrix}10/\\+--+----++\end{smallmatrix}$ &
$1^{3}2^{2}5^{-6}$&
$1^{3}2^{2}5^{-5}$&
$1^{3}2^{2}5^{-4}$\vphantom{$x^{x^{x^x}}$}&&
\\[2pt]
\hline
$\begin{smallmatrix}10/\\+-+-+--++-\end{smallmatrix}$ &
$1^{4}2^{-4}3^{-3}4^{-1}5^{-10}$\vphantom{$x^{x^{x^x}}$}&&&&
\\[2pt]
\hline
$\begin{smallmatrix}10/\\+-+0+0-+-+ \end{smallmatrix}$ &
$1^{4}2^{-2}4^{1}5^{-4}$\vphantom{$x^{x^{x^x}}$}&&&&
\\[2pt]
\hline
$\begin{smallmatrix}10/\\+-++-0-++0\end{smallmatrix}$ &
$1^{4}2^{-1}5^{-4}$\vphantom{$x^{x^{x^x}}$}&&&&
\\[2pt]
\hline
$\begin{smallmatrix}10/\\+-+-+-+0+0\end{smallmatrix}$ &
$1^{5}2^{-3}4^{1}5^{-1}$\vphantom{$x^{x^{x^x}}$}&&&&
\\[2pt]
\hline
$\begin{smallmatrix}10/\\+-+--+++--\end{smallmatrix}$ &
$1^{8}2^{-1}5^{-3}$&
$1^{8}2^{-1}5^{-2}$\vphantom{$x^{x^{x^x}}$}&&&
\\[2pt]
\hline
$\begin{smallmatrix}10/\\+-+-+-++--\end{smallmatrix}$ &
$1^{9}2^{-7}3^{-5}4^{2}5^{-9}$\vphantom{$x^{x^{x^x}}$}&&&&
\\[2pt]
\hline
$\begin{smallmatrix}10/\\+-+-+-+---\end{smallmatrix}$ &
$1^{9}2^{-7}3^{-4}4^{1}5^{-10}$&
$1^{9}2^{-6}3^{-3}5^{-8}$&
$1^{10}2^{-6}3^{-3}4^{1}5^{-5}$\vphantom{$x^{x^{x^x}}$}&&
\\[2pt]
\hline
$\begin{smallmatrix}12/\\+++-++-+++++\end{smallmatrix}$ &
$1^{-3}2^{8}3^{-1}4^{-4}$\vphantom{$x^{x^{x^x}}$}&&&&
\\[2pt]
\hline
$\begin{smallmatrix}12/\\++0-0+--0+0-\end{smallmatrix}$ &
$1^{-3}2^{9}3^{1}4^{-6}$\vphantom{$x^{x^{x^x}}$}&&&&
\\[2pt]
\hline
\multirow{2}{*}{$\begin{smallmatrix} 12/\\++--++-++--+\end{smallmatrix}$} &
$1^{-3}2^{10}3^{-3}4^{-8}$\vphantom{$x^{x^{x^x}}$}&
$1^{-2}2^{7}3^{-2}4^{-6}$&
$1^{-2}2^{8}3^{-2}4^{-7}$&
$1^{-2}2^{8}3^{-2}4^{-6}$&
$1^{-2}2^{9}3^{-2}4^{-7}$\\
&$1^{-2}2^{9}3^{-1}4^{-5}5^{1}$&
$1^{-2}2^{10}3^{-2}4^{-7}$&
$1^{-2}2^{10}3^{-1}4^{-5}5^{1}$&&
\\[2pt]
\hline
$\begin{smallmatrix}12/\\++--++--+--+\end{smallmatrix}$ &
$1^{-3}2^{10}3^{-2}4^{-7}5^{1}$\vphantom{$x^{x^{x^x}}$}&&&&
\\[2pt]
\hline
$\begin{smallmatrix}12/\\++-0+0-++0+0\end{smallmatrix}$ &
$1^{-2}2^{6}3^{-2}4^{-4}$\vphantom{$x^{x^{x^x}}$}&&&&
\\[2pt]
\hline
$\begin{smallmatrix}12/\\++--++-++---\end{smallmatrix}$ &
$1^{-2}2^{7}3^{-1}4^{-5}5^{1}$\vphantom{$x^{x^{x^x}}$}&&&&
\\[2pt]
\hline
$\begin{smallmatrix}12/\\++--+--++---\end{smallmatrix}$ &
$1^{-2}2^{7}3^{-1}4^{-4}5^{1}$\vphantom{$x^{x^{x^x}}$}&
$1^{-1}2^{6}3^{-1}4^{-5}$&
$1^{-1}2^{6}3^{-1}4^{-4}$&
$1^{-1}2^{7}3^{-1}4^{-4}$&
\\[2pt]
\hline
$\begin{smallmatrix}12/\\++--++-0+--0\end{smallmatrix}$ &
$1^{-2}2^{8}3^{-2}4^{-8}$\vphantom{$x^{x^{x^x}}$}&&&&
\\[2pt]
\hline
$\begin{smallmatrix}12/\\++-++--+--+-\end{smallmatrix}$ &
$1^{-1}2^{4}3^{-2}4^{-3}5^{1}$\vphantom{$x^{x^{x^x}}$}&&&&
\\[2pt]
\hline
$\begin{smallmatrix}12/\\++-++--++-++\end{smallmatrix}$ &
$1^{-1}2^{5}3^{-3}4^{-6}$\vphantom{$x^{x^{x^x}}$}&&&&
\\[2pt]
\hline
$\begin{smallmatrix}12/\\++-++-0+-0+-\end{smallmatrix}$ &
$1^{-1}2^{5}3^{-3}4^{-3}$\vphantom{$x^{x^{x^x}}$}&&&&
\\[2pt]
\hline
$\begin{smallmatrix}12/\\++--+--+++--\end{smallmatrix}$ &
$1^{-1}2^{6}3^{-2}4^{-5}5^{-2}$\vphantom{$x^{x^{x^x}}$}&&&&
\\[2pt]
\hline
$\begin{smallmatrix}12/\\++--+--++--+\end{smallmatrix}$ &
$1^{-1}2^{8}3^{-4}4^{-9}5^{-3}$\vphantom{$x^{x^{x^x}}$}&
$1^{-1}2^{8}3^{-3}4^{-8}5^{-2}$&
$1^{-1}2^{8}3^{-3}4^{-7}5^{-1}$&&
\\[2pt]
\hline
$\begin{smallmatrix}12/\\+--++-++-0+-\end{smallmatrix}$ &
$1^{1}3^{-3}4^{-4}$\vphantom{$x^{x^{x^x}}$}&&&&
\\[2pt]
\hline
\multirow{4}{*}{$\begin{smallmatrix}12/\\+--+--++-++-\end{smallmatrix}$} &
$1^{1}3^{-2}4^{-1}5^{3}$\vphantom{$x^{x^{x^x}}$}&
$1^{2}3^{-4}4^{-5}$&
$1^{2}3^{-4}4^{-4}5^{1}$&
$1^{2}3^{-3}4^{-3}5^{2}$&
$1^{2}3^{-2}4^{-2}5^{3}$\\
&$1^{2}2^{1}3^{-4}4^{-4}$&
$1^{2}2^{1}3^{-3}4^{-4}$&
$1^{2}2^{2}3^{-3}4^{-3}$&
$1^{3}2^{1}3^{-4}4^{-5}$&
$1^{3}2^{2}3^{-4}4^{-4}$\\
&$1^{3}2^{2}3^{-3}4^{-4}$&
$1^{3}2^{3}3^{-3}4^{-3}$&
$1^{3}2^{3}3^{-2}4^{-3}$&
$1^{4}2^{3}3^{-3}4^{-4}$&
$1^{4}2^{4}3^{-3}4^{-3}$\\
&$1^{4}2^{4}3^{-2}4^{-3}$&&&&
\\[2pt]
\hline
\multirow{2}{*}{$\begin{smallmatrix}12/\\+--++-++--++\end{smallmatrix}$} &
$1^{1}2^{1}3^{-3}4^{-6}$\vphantom{$x^{x^{x^x}}$}&
$1^{1}2^{1}3^{-3}4^{-5}$&
$1^{1}2^{2}3^{-3}4^{-5}$&
$1^{2}2^{2}3^{-2}4^{-4}$&
$1^{2}2^{3}3^{-2}4^{-5}$\\
&$1^{2}2^{3}3^{-2}4^{-4}$&
$1^{2}2^{4}3^{-2}4^{-4}$&
$1^{2}2^{5}3^{-2}4^{-4}$&&
\\[2pt]
\hline
\multirow{3}{*}{$\begin{smallmatrix} 12/\\+--++-++-+--\end{smallmatrix}$} &
$1^{1}2^{1}3^{-3}4^{-4}5^{-1}$\vphantom{$x^{x^{x^x}}$}&
$1^{1}2^{2}3^{-4}4^{-4}5^{-2}$&
$1^{1}2^{2}3^{-3}4^{-4}5^{-2}$&
$1^{1}2^{2}3^{-3}4^{-3}5^{-1}$&
\\
&$1^{1}2^{2}3^{-2}4^{-3}5^{-1}$&$1^{1}2^{2}3^{-2}4^{-2}$&
$1^{2}2^{3}3^{-4}4^{-5}5^{-3}$&
$1^{2}2^{3}3^{-3}4^{-4}5^{-2}$&
\\
&$1^{2}2^{3}3^{-2}4^{-3}5^{-1}$&$1^{2}2^{4}3^{-5}4^{-7}5^{-6}$&&&
\\[2pt]
\hline
$\begin{smallmatrix}12/\\+--+0-++-+0-\end{smallmatrix}$ &
$1^{1}2^{1}3^{-3}4^{-3}$&
$1^{1}2^{1}3^{-2}4^{-2}5^{1}$&
$1^{1}2^{2}3^{-3}4^{-2}$\vphantom{$x^{x^{x^x}}$}&&
\\[2pt]
\hline
$\begin{smallmatrix}12/\\+--++--+--++\end{smallmatrix}$ &
$1^{1}2^{3}3^{-3}4^{-7}$&
$1^{1}2^{4}3^{-3}4^{-6}$&
$1^{1}2^{4}3^{-2}4^{-5}5^{1}$&
$1^{2}2^{5}3^{-2}4^{-6}$\vphantom{$x^{x^{x^x}}$}&
\\[2pt]
\hline
\multirow{2}{*}{$\begin{smallmatrix}12/\\+--+--+--++-\end{smallmatrix}$} &
$1^{2}3^{-4}4^{-3}5^{2}$\vphantom{$x^{x^{x^x}}$}&
$1^{2}3^{-3}4^{-2}5^{3}$&
$1^{2}2^{1}3^{-4}4^{-3}5^{1}$&
$1^{3}2^{1}3^{-5}4^{-5}$&
\\
&$1^{3}2^{1}3^{-4}4^{-4}5^{1}$&$1^{3}2^{2}3^{-3}4^{-2}5^{2}$&$1^{4}2^{3}3^{-4}4^{-4}$&&
\\[2pt]
\hline
$\begin{smallmatrix}12/\\+--+--++---+\end{smallmatrix}$ &
$1^{2}3^{-3}4^{-4}5^{-1}$&
$1^{3}2^{1}3^{-2}4^{-3}5^{-1}$&
$1^{3}2^{2}3^{-2}4^{-3}5^{-1}$\vphantom{$x^{x^{x^x}}$}&&
\\[2pt]
\hline
$\begin{smallmatrix}12/\\+--+--++--++\end{smallmatrix}$ &
$1^{2}2^{1}3^{-4}4^{-6}5^{-2}$&
$1^{3}2^{2}3^{-6}4^{-9}5^{-5}$&
$1^{3}2^{3}3^{-3}4^{-5}5^{-2}$\vphantom{$x^{x^{x^x}}$}&&
\\[2pt]
\hline
$\begin{smallmatrix}12/\\+--++-++---+\end{smallmatrix}$ &
$1^{2}2^{1}3^{-2}4^{-4}$\vphantom{$x^{x^{x^x}}$}&&&&
\\[2pt]
\hline
$\begin{smallmatrix}12/\\+--++-++--+-\end{smallmatrix}$ &
$1^{2}2^{3}3^{-6}4^{-9}5^{-4}$\vphantom{$x^{x^{x^x}}$}&&&&
\\[2pt]
\hline
\multirow{2}{*}{$\begin{smallmatrix}12/\\+-++--+--+-+\end{smallmatrix}$} &
$1^{3}2^{-1}3^{-3}4^{-1}$\vphantom{$x^{x^{x^x}}$}&
$1^{3}2^{-1}3^{-2}4^{-1}$&
$1^{4}2^{-1}3^{-3}4^{-2}$&
$1^{4}2^{-1}3^{-2}4^{-1}5^{1}$&\\
&$1^{4}3^{-4}4^{-4}5^{-3}$
&$1^{4}3^{-4}4^{-3}5^{-2}$&
$1^{4}3^{-3}4^{-3}5^{-2}$&
$1^{4}3^{-3}4^{-2}5^{-1}$&
\\[2pt]
\hline
$\begin{smallmatrix}12/\\+-0+0-0+0-0+\end{smallmatrix}$ &
$1^{3}3^{-1}4^{-3}$\vphantom{$x^{x^{x^x}}$}&&&&
\\[2pt]
\hline
$\begin{smallmatrix}12/\\+--+--+--+0-\end{smallmatrix}$ &
$1^{3}2^{2}3^{-5}4^{-4}$&
$1^{4}2^{4}3^{-4}4^{-3}$\vphantom{$x^{x^{x^x}}$}&&&
\\[2pt]
\hline
$\begin{smallmatrix}12/\\+--+--++--+-\end{smallmatrix}$ &
$1^{3}2^{3}3^{-4}4^{-6}5^{-2}$\vphantom{$x^{x^{x^x}}$}&&&&
\\[2pt]
\hline
\multirow{2}{*}{$\begin{smallmatrix}12/\\+-++--+--++-\end{smallmatrix}$} &
$1^{4}2^{-1}3^{-3}4^{-3}5^{1}$\vphantom{$x^{x^{x^x}}$}&
$1^{4}2^{-1}3^{-3}4^{-3}5^{2}$&
$1^{5}3^{-5}4^{-6}5^{-1}$&
$1^{6}2^{2}3^{-4}4^{-5}5^{-2}$&\\
&$1^{6}2^{2}3^{-4}4^{-5}5^{-1}$&&&&
\\[2pt]
\hline
$\begin{smallmatrix}12/\\+-++-0+--+-0\end{smallmatrix}$ &
$1^{4}2^{-1}3^{-4}4^{-2}$\vphantom{$x^{x^{x^x}}$}&&&&
\\[2pt]
\hline
$\begin{smallmatrix}12/\\+-++--+--+--\end{smallmatrix}$ &
$1^{5}3^{-3}4^{-2}$\vphantom{$x^{x^{x^x}}$}&&&&
\\[2pt]
\hline
\multirow{2}{*}{$\begin{smallmatrix}12/\\+-++-++--+--\end{smallmatrix}$} &
$1^{5}3^{-4}4^{-2}$\vphantom{$x^{x^{x^x}}$}&
$1^{6}3^{-4}4^{-3}$&
$1^{6}3^{-3}4^{-2}5^{1}$&
$1^{6}2^{1}3^{-4}4^{-2}$&
$1^{6}2^{1}3^{-3}4^{-2}$\\
&$1^{7}3^{-4}4^{-4}$&
$1^{7}2^{1}3^{-4}4^{-3}$&
$1^{7}2^{1}3^{-3}4^{-3}$&
$1^{8}2^{2}3^{-3}4^{-3}$&
\\[2pt]
\hline
\multirow{2}{*}{$\begin{smallmatrix}12/\\+-++-++-++--\end{smallmatrix}$} &
$1^{6}2^{-2}3^{-6}4^{-5}$\vphantom{$x^{x^{x^x}}$}&
$1^{6}2^{-2}3^{-5}4^{-4}5^{1}$&
$1^{6}2^{-2}3^{-4}4^{-3}5^{2}$&
$1^{6}2^{-1}3^{-5}4^{-3}5^{1}$&
\\
&$1^{6}2^{-1}3^{-4}4^{-2}5^{2}$
&$1^{7}3^{-5}4^{-4}$&
$1^{7}2^{1}3^{-5}4^{-3}$&&
\\[2pt]
\hline
$\begin{smallmatrix}12/\\+-++-++--+-+\end{smallmatrix}$ &
$1^{6}2^{1}3^{-3}4^{-3}5^{-1}$\vphantom{$x^{x^{x^x}}$}&&&&
\\[2pt]
\hline
$\begin{smallmatrix}12/\\+-+--++-++--\end{smallmatrix}$ &
$1^{7}2^{-3}3^{-4}4^{-3}5^{1}$\vphantom{$x^{x^{x^x}}$}&&&&
\\[2pt]
\hline
$\begin{smallmatrix}12/\\+-+--++-+++-\end{smallmatrix}$ &
$1^{7}2^{-3}3^{-3}4^{-4}$&
$1^{7}2^{-3}3^{-3}4^{-3}$&
$1^{7}2^{-3}3^{-3}4^{-2}$\vphantom{$x^{x^{x^x}}$}&&
\\[2pt]
\hline
$\begin{smallmatrix}12/\\+-+--++-++-+\end{smallmatrix}$ &
$1^{8}2^{-3}3^{-4}4^{-2}$\vphantom{$x^{x^{x^x}}$}&&&&
\\[2pt]
\hline
\multirow{2}{*}{$\begin{smallmatrix}12/\\+-+--+--++-+\end{smallmatrix}$} &
$1^{8}2^{-3}3^{-5}4^{-2}$\vphantom{$x^{x^{x^x}}$}&
$1^{9}2^{-3}3^{-5}4^{-3}$&
$1^{9}2^{-3}3^{-4}4^{-2}5^{1}$&
$1^{9}2^{-2}3^{-5}4^{-2}$&
$1^{9}2^{-2}3^{-4}4^{-2}$\\
&$1^{10}2^{-2}3^{-5}4^{-3}$&
$1^{10}2^{-2}3^{-4}4^{-2}5^{1}$&
$1^{10}3^{-5}4^{-3}5^{-2}$&&
\\[2pt]
\hline
$\begin{smallmatrix}12/\\+-+-+++-+0+-\end{smallmatrix}$ &
$1^{9}2^{-4}3^{-3}4^{-4}$\vphantom{$x^{x^{x^x}}$}&&&&
\\[2pt]
\hline
$\begin{smallmatrix}12/\\+-+-+++-+-+-\end{smallmatrix}$ &
$1^{9}2^{-4}3^{-3}4^{-3}$&
$1^{9}2^{-4}3^{-3}4^{-2}$\vphantom{$x^{x^{x^x}}$}&&&
\\[2pt]
\hline
$\begin{smallmatrix}12/\\+-+--+--+--+\end{smallmatrix}$ &
$1^{10}2^{-4}3^{-6}4^{-5}$&
$1^{10}2^{-3}3^{-6}4^{-4}$&
$1^{10}2^{-2}3^{-6}4^{-3}$&
$1^{10}2^{-2}3^{-5}4^{-2}5^{1}$\vphantom{$x^{x^{x^x}}$}&
\\[2pt]
\hline
$\begin{smallmatrix}15/\\+0-+0-+--+-0+-0\end{smallmatrix}$ &
$2^{3}3^{-2}5^{-3}$\vphantom{$x^{x^{x^x}}$}&&&&
\\[2pt]
\hline
$\begin{smallmatrix}15/\\+--+-0+-0+-0+-0\end{smallmatrix}$ &
$1^{1}3^{-3}5^{-2}$\vphantom{$x^{x^{x^x}}$}&&&&
\\[2pt]
\hline
\end{longtable}
	
\end{document}